\newif\ifpictures
\picturestrue

\documentclass[12pt]{amsart}
\usepackage{amssymb,amsmath}
 \usepackage{amsopn}
 \usepackage{xspace}
  \usepackage{hyperref}
 \usepackage[dvips]{graphicx}
\usepackage[arrow,matrix,curve]{xy}
\usepackage {color, tikz}
\usepackage{paralist}

\numberwithin{equation}{section}
\newtheorem{thm}{Theorem}
\newtheorem{example}[thm]{Example}
\newtheorem{prop}[thm]{Proposition}
\newtheorem{lemma}[thm]{Lemma}
\newtheorem{cor}[thm]{Corollary}

\newtheorem{definition}[thm]{Definition}
\newtheorem{Bezout}[thm]{Amoeba B{\'e}zout Theorem}
\newtheorem{Bernstein}[thm]{Amoeba Bernstein Theorem}
\numberwithin{thm}{section}

\newcounter{FNC}[page]
\def\newfootnote#1{{\addtocounter{FNC}{2}$^\fnsymbol{FNC}$%
     \let\thefootnote\relax\footnotetext{$^\fnsymbol{FNC}$#1}}}

\newcommand{\struc}[1]{{\color{blue} #1}}

\definecolor{DarkGreen}{rgb}{0,0.65,0}

\newcommand{\C}{\mathbb{C}}
\newcommand{\N}{\mathbb{N}}

\newcommand{\R}{\mathbb{R}}

\newcommand{\Z}{\mathbb{Z}}

\newcommand{\lf}{\left}
\newcommand{\ri}{\right}
\newcommand{\ra}{\rightarrow}

\newcommand{\ovl}{\overline}

\newcommand\cA{{\ensuremath{\mathcal{A}}}\xspace}

\newcommand\cC{{\ensuremath{\mathcal{C}}}\xspace}

\newcommand\cF{{\ensuremath{\mathcal{F}}}\xspace}

\newcommand\cI{{\ensuremath{\mathcal{I}}}\xspace}

\newcommand\cO{{\ensuremath{\mathcal{O}}}\xspace}

\newcommand\cS{{\ensuremath{\mathcal{S}}}\xspace}
\newcommand\cT{{\ensuremath{\mathcal{T}}}\xspace}
\newcommand\cU{{\ensuremath{\mathcal{U}}}\xspace}
\newcommand\cV{{\ensuremath{\mathcal{V}}}\xspace}

\newcommand{\eps}{\varepsilon}

\newcommand{\alp}{\alpha}
\newcommand{\lam}{\lambda}

\newcommand{\sig}{\sigma}

\DeclareMathOperator{\conv}{conv}

\DeclareMathOperator{\trop}{trop}

\DeclareMathOperator{\Log}{Log}

\DeclareMathOperator{\New}{New}

\DeclareMathOperator{\Vol}{Vol}
\DeclareMathOperator{\re}{re}
\DeclareMathOperator{\im}{im}

\DeclareMathOperator{\ord}{ord}

\DeclareMathOperator{\Mink}{Mink}

\DeclareMathOperator{\NF}{NF}

\DeclareMathOperator{\Int}{int}

\DeclareMathOperator{\MV}{MV}

\title[Intersections of Amoebas]{Intersections of Amoebas}
\author{Martina Juhnke-Kubitzke}
\address{Martina Juhnke-Kubitzke, Universit\"at Osnabr\"uck, FB 6 -- Institut f\"ur Mathematik, Albrechtstr. 28a,
 49076 Osnabr\"uck, Germany}
 \email{juhnke-kubitzke@uni-osnabrueck.de}

\medskip

\author{Timo de Wolff}
\address{Timo de Wolff, Texas A\&M University, Department of Mathematics, College Station, TX 77843-3368, 
 USA}
 \email{dewolff@math.tamu.edu}


\begin{document}

\begin{abstract}
Amoebas are projections of complex algebraic varieties in the algebraic torus under a Log-absolute value map, which have connections to various mathematical subjects. While amoebas of hypersurfaces have been intensively studied in recent years, the non-hypersurface case is barely understood so far. 

We investigate intersections of amoebas of $n$ hypersurfaces in $(\mathbb{C}^*)^n$, which are canonical supersets of amoebas given by non-hypersurface varieties. Our main results are amoeba analogs of \emph{Bernstein's Theorem} and \emph{B{\'e}zout's Theorem} providing an upper bound for the number of connected components of such intersections. Moreover, we show that the \emph{order map} for hypersurface amoebas can be generalized in a natural way to intersections of amoebas. In particular, analogous to the case of amoebas of hypersurfaces, the restriction of this generalized order map to a single connected component is still $1$-to-$1$.
\end{abstract}

\maketitle

\section{Introduction}
 Let \struc{$I$} be an ideal generated by finitely many \struc{\textit{Laurent polynomials}} $\struc{f_1,\ldots,f_k} \in \struc{\C[\mathbf{z}^{\pm 1}]} := \C[z_1^{\pm 1},\ldots,z_n^{\pm 1}]$ and let  $\struc{\cV(I)} \subseteq \struc{(\C^*)^n} := (\C \setminus \{0\})^n$ be the  corresponding variety. The \textit{\struc{amoeba}} \struc{$\cA(I)$} of $I$, as originally defined by Gelfand, Kapranov, and Zelevinsky \cite{Gelfand:Kapranov:Zelevinsky}, is the image of $\cV(I)$ under the $\Log$-absolute map given by 
\begin{eqnarray}
\label{Equ:LogMap}
	\struc{\Log|\cdot|}: \ \lf(\C^*\ri)^n \to \R^n, \quad (z_1,\ldots,z_n) \mapsto (\log|z_1|,
\ldots, \log|z_n|) \, .
\end{eqnarray}
If $I = \langle f \rangle$, then we write \struc{$\cA(f)$} for simplicity.\\
Amoebas became prominent during the last twenty years since they provide a natural connection between algebraic geometry and tropical geometry; see \cite{deWolff:AmoebaTropicalizationSurvey,Maclagan:Sturmfels,Mikhalkin:Survey} for an overview. Furthermore, amoebas are objects with a rich structure themselves and there exist connections to numerous other mathematical subjects like complex analysis \cite{Forsberg:Passare:Tsikh}, nonnegativity of real polynomials \cite{Iliman:deWolff:Circuits}, crystal shapes \cite{Kenyon:Okounkov:Sheffield}, the topology of real curves \cite{Mikhalkin:Annals}, and statistical thermodynamics \cite{Passare:Pochekutov:Tsikh}. For an overview about amoeba theory see \cite{deWolff:Diss,deWolff:AmoebaTropicalizationSurvey,Mikhalkin:Survey,Passare:Tsikh:Survey,Rullgard:Diss}.

Though amoebas of hypersurfaces, i.e., amoebas associated to a single Laurent polynomial, have been intensively studied in recent years, the non-hypersurface case is still barely understood. For almost all properties that are true in the hypersurface case it is not known if they still hold for arbitrary varieties. One exception are statements regarding convexity, which were recently shown by Nisse and Sottile; see \cite{Nisse:Sottile:HigherConvexity}. Another one of the few results concerning amoebas of ideals that are not principal was shown by Purbhoo. It states that the amoeba of an arbitrary ideal $I \subseteq \C[\mathbf{z}^{\pm 1}]$ can be written as the intersection of the amoebas of all the elements of $I$; see \cite[Corollary 5.6.]{Purbhoo}:
\begin{eqnarray}
 \cA(I) & = & \bigcap_{f \in I} \cA(f). \label{Equ:PurbhooAmoebaIdeals}
\end{eqnarray}

Since this result is not useful from a computational point of view, the question about the existence of an \struc{\textit{amoeba basis}} arose in the article \cite{Schroeter:deWolff} by Schroeter and the second author. Here, an amoeba basis refers to a finite set of Laurent polynomials $f_1,\ldots,f_k$ such that $\langle f_1,\ldots,f_k \rangle = I$ and $\cA(I) = \bigcap_{j = 1}^k \cA(f_k)$. So far, amoeba bases are known to exist only in very special cases \cite{Nisse:AmoebaBasesZeroDimensional,Schroeter:deWolff} while a recent result by Nisse claims non-existence in general \cite{Nisse:AmoebaBasesNonExistence}. By \eqref{Equ:PurbhooAmoebaIdeals}, however, the inclusion $\cA(I) \subseteq \bigcap_{j = 1}^k \cA(f_k)$ holds  for \textit{every} collection $f_1,\ldots,f_k$ with $\langle f_1,\ldots,f_k \rangle \subseteq I$ and it is reasonable to expect that information about  $\bigcap_{j = 1}^k \cA(f_k)$ also provides information about $\cA(I)$. 
Phrased differently, understanding finite intersections of hypersurface amoebas is an essential interim stage for understanding amoebas of arbitrary ideals. Another key advantage of this approach is that finite intersections of hypersurface amoebas turn out to be much more accessible than amoebas of arbitrary ideals. This serves as the key motivation for this article. 

\bigskip
Let $\struc{\cF}:=\{f_1,\ldots,f_n\}\subseteq\C[\mathbf{z}^{\pm 1}]$ be a collection of $n$-variate complex Laurent polynomials. In this article, we study the intersection of their corresponding amoebas $\cA(f_1),\ldots,$ $\cA(f_n)\subseteq\R^n$. 

We show that intersections $\struc{\cI(\cF)}:=\bigcap_{j = 1}^n \cA(f_j)$ preserve a significant amount of the amoeba structure from the hypersurface case. Moreover, these intersections carry an additional interesting and rich combinatorial structure on their own.

\begin{figure}
\ifpictures
\includegraphics[width=0.4\linewidth]{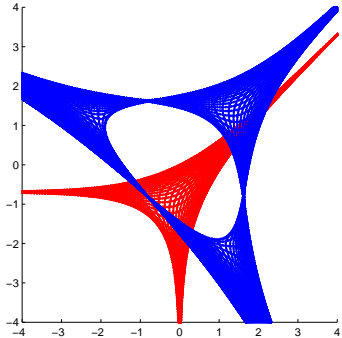}
\fi
\caption{An approximation of the amoebas $\cA(f_1)$ and $\cA(f_2)$ for $f_1 := 2z_1 + z_2 + 1$ and $f_2 := z_1^2z_2 + z_1z_2^2 + 5 z_1z_2 + 1$. The intersection $\cA(f_1)\cap\cA(f_2)$ consists of two connected components.}
\label{Fig:AmoebaIntersection1}
\end{figure}

Since, in general, $\cI(\cF)$ consists of several connected components, we focus on the study of the combinatorics of these components; see Figure \ref{Fig:AmoebaIntersection1} and Section \ref{Sec:Combinatorics}. In particular, the convex hull of each of these components is proven to be a simple polytope, which, in the sequel, will be referred to as \struc{\textit{intersection polytope}}. A natural question is, how many connected components such an intersection $\cI(\cF)$ can have. Our first main result provides an upper bound for this quantity.  More precisely, as an analogue to the classical Bernstein Theorem, we show the following \struc{\textit{Amoeba Bernstein Theorem}}; see Theorem \ref{Thm:AmoebaBernstein} for the detailed version.

\begin{thm}
Let $\cF=\{f_1,\ldots,f_n\}\subseteq \C[\mathbf{z}^{\pm 1}]$ be a generic collection of Laurent polynomials. The number of connected components of $\cI(\cF)$ is bounded from above by the mixed volume $\MV(\New(f_1),\ldots,\New(f_n))$ of the Newton polytopes of $f_1,\ldots,f_n$.  
Moreover, this bound is optimal in the following sense: If one deformation retracts every amoeba $\cA(f_j)$ in $\cI(\cF)$ to its spine $\cS(f_j)$, then the number of connected components of the intersection of the deformed amoebas converges to $\MV(\New(f_1),\ldots,\New(f_n))$. 
\end{thm}

The term ``generic collection of Laurent polynomials'' will be made precise in Section \ref{Sec:Combinatorics}.  We remark that this theorem is highly non-obvious since it is unclear whether every connected component of $\cI(\cF)$ contains a point, that is the projection of a point in $\cV(\langle f_1,\ldots,f_n \rangle)$ with respect to the $\Log|\cdot|$-map. Therefore, no obvious upper bound for the number of connected components can be obtained directly from the classical Bernstein Theorem. As an immediate consequence we obtain the following \struc{\textit{Amoeba B{\'e}zout Theorem}}; see Theorem \ref{Cor:AmoebaBezout} for the detailed version.

\begin{thm}
Let $\mathcal{F}=\{f_1,\ldots,f_n\} \subseteq \C[\mathbf{z}^{\pm 1}]$ be a generic collection of Laurent polynomials. The number of connected components of $\cI(\cF)$ is smaller than  or equal to the product of the total degrees of the $f_j$.
\label{Thm:AmoebaBezout}
\end{thm}

In Section \ref{Sec:OrderMap}, we construct a generalization of the \struc{\textit{order map}} of hypersurface amoebas to our setting. The usual order map, introduced by Forsberg, Passare, and Tsikh \cite{Forsberg:Passare:Tsikh}, relates the components of the complement of a hypersurface amoeba $\cA(f)$ to the lattice points in the corresponding Newton polytope $\New(f)$; see Section \ref{SubSec:OrderMap} for details. Given a collection $\cF$ of $n$ Laurent polynomials, we define a natural generalization of the order map to the vertex sets of the intersection polytopes of $\cF$ and thereby also to the polytope which is the convex hull of the entire intersection $\cI(\cF)$. 
We show that several properties of the order map for the hypersurface case are preserved in this more general setting; see Theorems \ref{Thm:OrderMapConnectivityComponents} and \ref{Thm:NormalFan}, as well as Corollary \ref{Cor:OrderMapInjectivityTotalConvexHull}. The next theorem summarizes those results. The notion of a \struc{\textit{mixed normal cone}}, which is used in this statement, will be explained in Section \ref{Sec:OrderMap}.

\begin{thm}
Let $\cF:=\{f_1,\ldots,f_n\} \subseteq \C[\mathbf{z}^{\pm 1}]$ be a generic collection of Laurent polynomials. Let $K$ be a  connected component of $\cI(\cF)$, let $P_K := \conv(K)$ be the corresponding intersection polytope  and let $P := \conv\left(\cI(\cF)\right)$. Then the following statements hold:
\begin{enumerate}
 \item[(a)] There exists a \textit{generalized order map} from the vertices of $P_K$ and $P$, respectively, to $(\New(f_1) \times \cdots \times \New(f_n)) \cap \Z^{n \times n}$, that is injective on $P_K$ and on $P$, respectively.
 \item[(b)] The vertices of $P$ are in $1$-to-$1$-correspondence with those vertices of the Minkowski sum $\New(f_1) + \cdots + \New(f_n)$ which have a \textit{mixed} normal cone.
\end{enumerate}
\end{thm}
\medskip

The article is organized as follows. In Section \ref{Sec:Preliminaries} we fix the essential notation and provide necessary background material. This includes facts from convex geometry and about amoebas, in particular about the order map and the spine. Moreover, we provide a short review of mixed volumes as well as of the classical and tropical versions of Bernstein's Theorem. In Section \ref{Sec:Combinatorics} we discuss combinatorial properties of intersections of hypersurface amoebas and of their associated intersection polytopes. The \textit{Amoeba Bernstein Theorem} is the main result of this section. In Section \ref{Sec:OrderMap} we define the generalized order map and discuss its properties, including the injectivity statements mentioned above.

\section*{Acknowledgements}
We thank Jens Forsg{\aa}rd and the anonymous referee for their helpful comments. \\
The first author was supported by the German Research Council DFG-GRK~1916. 
The second author was partially supported by GIF Grant no.\ 1174/2011, DFG project MA 4797/3-2 and DFG 
project TH 1333/2-1.

\section{Preliminaries}
\label{Sec:Preliminaries}

\subsection{Convex geometry} \label{SubSec:Polytopes}
Throughout most parts of this article, we assume that the reader has some basic knowledge from discrete and convex geometry. For background information we recommend \cite{Ziegler:Book} as a reference. 
As a service to the reader, we recall notions of some subjects that are less well-known and of particular importance for this article. 
A set $C\subseteq \R^n$ is called \struc{\textit{strictly convex}} if $C$ is convex and for any distinct points $p,q$ in the boundary $\partial ( C)$ of $C$ the line segment through $p$ and $q$ intersects $\partial (C)$ only in $p$ and $q$. In particular, a strictly convex set does not contain any line segment in its boundary. We will use that the intersection of strictly convex sets is always strictly convex. 
We recall that an \struc{\textit{extreme point}} $p$ of a convex set $C$ is a point $p\in C$ such that, whenever $p=\lambda s+(1-\lambda)t $ for points $s,t\in C$ and $0\leq \lambda\leq 1$, it must hold that $s=p$ or $t=p$. We will use in Section \ref{Sec:Combinatorics} that, by the Krein-Milman theorem, any compact and convex set in $\R^n$ is the convex hull of its extreme points \cite{Krein:Milman}. 
For a set $S\subseteq \R^n$ (not necessarily convex), the \struc{\textit{relative interior}} of $S$ is the interior of $S$ with respect to its affine hull. 

Given a Laurent polynomial $\struc{f}:=\sum_{a\in \Z^n}\lambda_a \mathbf{z}^a\in \C[\mathbf{z}^{\pm 1}]$, where $\struc{\mathbf{z}^a}:=z_1^{a_1}\cdots z_n^{a_n}$ for $\struc{a}:=(a_1,\ldots,a_n)\in \Z^n$,  the set $\struc{A} := \{a~:~\lambda_a\neq 0\}$ is called the \struc{\textit{support}} of $f$.  The \struc{\textit{Newton polytope}} $\New(f)$ of $f$ is the lattice polytope 
\begin{equation*}
\struc{\New(f)} \ := \ \conv(A) \ \subseteq \ \R^n.
\end{equation*}
In what follows, we refer to the set $\struc{(\C^*)^A}$ of all Laurent polynomials with support $A$ as the \struc{\textit{configuration space}} corresponding to $A$. For usual polynomials $f \in \C[\mathbf{z}]$ we consider analogously the configuration space $\struc{\C^A}$. 
Given a polytope $P\subseteq \R^n$, we are frequently interested in its normal fan, which is defined as follows: To every non-empty face $G\in P$ one associates the following \struc{\textit{normal cone}}
\begin{equation*}
\struc{\NF_G(P)} \ := \ \{\mathbf{c}\in \R^n~:~G\subseteq \{\mathbf{y}\in P~:~\langle \mathbf{c},\mathbf{y}\rangle=\max_{\mathbf{w}\in P}\langle \mathbf{c},\mathbf{w}\rangle\}\},
\end{equation*}
where $\langle \cdot,\cdot \rangle$ denotes the standard scalar product in $\R^n$. 
The collection of these cones is called the \struc{\emph{normal fan}} of $P$, denoted by $\NF(P)$:
\begin{equation*}
\struc{\NF(P)} \ := \ \{\NF_G(P)~:~\emptyset\neq G\subseteq P \mbox{ face of } P\}.
\end{equation*}
If $P:=\New(f)$ is the Newton polytope of a Laurent polynomial $f\in \C[\mathbf{z}^{\pm 1}]$ and $G$ is a face of $\New(f)$, then we write \struc{$\NF_G(f)$} for $\NF_G(\New(f))$. Similarly, we write \struc{$\NF(f)$} for the normal fan of $\New(f)$.

\subsection{The Order Map}
\label{SubSec:OrderMap}
Given a Laurent polynomial $f\in \C[\mathbf{z^{\pm 1}}]$, the complement of its amoeba $\cA(f)$ consists of several bounded or unbounded connected components. As shown in \cite{Forsberg:Passare:Tsikh}, there exists a close connection between those components and lattice points in the corresponding Newton polytope $\New(f)$. The precise relation is given via the \struc{\emph{order map}} introduced by Forsberg, Passare, and Tsikh \cite[Definition 2.1]{Forsberg:Passare:Tsikh}:
\begin{eqnarray}
  \label{Equ:Ordermap}
\struc{\ord}: \R^n \setminus \cA(f) & \to & \R^n, \quad \mathbf{w} \mapsto (u_1,\ldots,u_n) \ \text{ with } \\
  u_j & := & \frac{1}{(2\pi i)^n} \int_{\Log|\mathbf{z}| = \mathbf{w}} \frac{z_j \partial_j f(\mathbf{z})}{f(\mathbf{z})}
    \frac{dz_1 \cdots dz_n}{z_1 \cdots z_n} \ \text{ for all } \ 1 \le j \le n \, .\nonumber
\end{eqnarray}

We write \struc{$\im(\ord(f))$} for the image of $\R^n \setminus \cA(f)$ under the order map. 

The order map can be understood as a multivariate analogue of the classical \textit{argument principle} from complex analysis. The following theorem by Forsberg, Passare and Tsikh \cite[Propositions 2.4 and 2.5]{Forsberg:Passare:Tsikh} provides the indicated connection between components of the complement of $\cA(f)$ and lattice points in $\New(f)$. 
\begin{thm}
Given a Laurent polynomial $f\in\C[\mathbf{z}^{\pm 1}]$ the image of the order map $\im(\ord(f))$ is contained in $\New(f) \cap \Z^n$. If $\mathbf{w},\mathbf{w}' \in \R^n\setminus \cA(f)$, then $\mathbf{w}$ and $\mathbf{w}'$ belong to the same component of the complement of $\cA(f)$ if and only if $\ord(\mathbf{w}) = \ord(\mathbf{w}')$.
\label{Thm:OrderMap}
\end{thm}
By Theorem \ref{Thm:OrderMap}, every component of the complement of a given amoeba $\cA(f)$ corresponds to a unique lattice point in the Newton polytope $\New(f)$ of $f$.  We denote for each $\alp \in \New(f) \cap \Z^n$ its corresponding (possibly empty) component of the complement of $\cA(f)$ by $E_{\alpha}(f)$, i.e., 
\begin{eqnarray*}
	\struc{E_{\alpha}(f)} & := & \{\mathbf{w} \in \R^n \setminus \cA(f) \ : \ \ord(\mathbf{w}) = \alpha\}. \label{Equ:DefEalp}
\end{eqnarray*}
Points in $E_{\alpha}(f)$ are said to be of \struc{\textit{order} $\alpha$} and $E_\alp(f)$ is called \struc{\textit{the component of order} $\alp$} of the complement.\\

Besides the just described connection between components of the complement of $\cA(f)$ and lattice points in $\New(f)$, Gelfand, Kapranov and Zelevinsky \cite[Chapter 6]{Gelfand:Kapranov:Zelevinsky} showed that the connected components of the complement are also linked to the normal fan of $\New(f)$. Indeed, the precise relation is the following one:

\begin{thm}
Let $f \in \C[\mathbf{z}^{\pm 1}]$ be a Laurent polynomial with support set $A \subseteq \Z^n$. The set of vertices of $\New(f)$ is in bijective correspondence with a subset of the components of the complement of $\cA(f)$. Namely, let $\alp \in A$ be a vertex of $\New(f)$ and  $\NF_{\alp}(f)$ be the corresponding cone in $\NF(f)$. Then there exists a unique non-empty, unbounded component $E_{\alp}(f)$ in the complement of $\cA(f)$ that contains an affine translation of $\NF_{\alp}(f)$.
\label{Thm:GKZAmoebaVertex}
\end{thm}

\subsection{Tropical Geometry and the Spine}
\label{SubSec:Tropical}
Looking at the image of an amoeba, one immediately observes that it has finitely many ``\textit{tentacles}'', which point in different directions. These tentacles point towards a set of points at infinity, which is called the \textit{logarithmic limit set}. This set was introduced by Bergman \cite{Bergman}. For a Laurent polynomial $f$ with amoeba $\cA(f)$ and any positive real number $r\in \R$ one defines 
\begin{eqnarray}
	\struc{\cA_r(f)}	& :=	& (1/r \cdot \cA(f)) \cap \mathbb{S}^{n-1}. \label{Equ:ScaledAmoeba}
\end{eqnarray}
Here, $1/r \cdot \cA(f) := \{1/r \cdot \mathbf{w} : \mathbf{w} \in \cA(f)\}$ and $\mathbb{S}^{n-1}$ denotes the $(n-1)$-dimensional unit sphere $\struc{\mathbb{S}^{n-1}} := \{\mathbf{w} \in \R^n \ : \ ||\mathbf{w}||_2 = 1\}$. The \struc{\textit{logarithmic limit set}} $\cA_\infty(f)$ is defined as 
\begin{eqnarray*}
 \struc{\cA_\infty(f)} & := & \lim_{r \ra \infty} \cA_r(f).
\end{eqnarray*}
Bieri and Groves showed that $\cA_\infty(f)$ is a rational, polyhedral fan on the unit sphere; \cite{Bieri:Groves}, see also \cite[Section 1.4]{Maclagan:Sturmfels}. Moreover, amoebas $\cA(f)$ are unbounded, and their complements are open. These facts will be used in Section \ref{SubSec:Tropical}.\\

In what follows, we introduce the \textit{spine} of an amoeba $\cA(f)$. In Section \ref{SubSec:Tropical}, we repeatedly use that the spine is both, a piecewise linear deformation retract of $\cA(f)$ and a tropical hypersurface.

First, we need to recall some notion from tropical geometry. For additional background on tropical geometry we refer to \cite{Maclagan:Sturmfels}. 
The \struc{\textit{tropical semiring} $(\R \cup \{-\infty\},\oplus,\odot)$} is defined by the operations
\begin{eqnarray*}
	\struc{a \oplus b} \ := \ \max\{a,b\}, & \text{ and } & \struc{a \odot b} \ := \ a + b.
\end{eqnarray*}
Note that $-\infty$ is the neutral element for the tropical addition $\oplus$. We remark that some authors prefer the minimum together with $+ \infty$ instead of the maximum  together with $-\infty$ as tropical addition. A \struc{\textit{tropical monomial}} is a function
\begin{eqnarray*}
	(\R \cup \{-\infty\})^n \ra \R \cup \{-\infty\}, \quad (z_1,\ldots,z_n) \mapsto \struc{b_a \odot \mathbf{z}^{a}} \ := \ b_a \odot z_1^{a_1} \odot \cdots \odot z_n^{a_n}
\end{eqnarray*}
with $b_a \in \R$ and $a:=(a_1,\ldots,a_n) \in \N^n$. In terms of classical operations a tropical monomial is the affine linear form $b_a + \langle \mathbf{z},a \rangle$. A \struc{\textit{tropical polynomial}} with \struc{\textit{support set}} $\struc{A} \subseteq \N^n$ is a finite tropical sum of tropical monomials, i.e., it is a function 
\begin{eqnarray*}
	(\R \cup \{-\infty\})^n \ra \R \cup \{-\infty\}, \quad (z_1,\ldots,z_n) \mapsto \struc{\bigoplus_{a \in A} b_a \odot \mathbf{z}^{a}} \ := \ \max_{a \in A} \{b_a + \langle \mathbf{z},a \rangle\},
\end{eqnarray*}
where $b_a\in \R\setminus \{0\}$. 
For a tropical polynomial $h$ as above its \struc{\textit{tropical hypersurface}} or \struc{\textit{tropical variety} $\cT(h)$} is defined as the set of points $\mathbf{x}$ in $(\R\cup\{-\infty\})^n$ such that the maximum of $\{b_a + \langle \mathbf{x},a \rangle~:~a\in A\}$ is attained at least twice. The tropical hypersurface $\cT(h)$ is a polyhedral complex, which is dual to  a (regular) subdivision of the Newton polytope of $h$.\\

The definition of the spine of an amoeba requires the definition of the \textit{Ronkin function} \cite{Ronkin}; see also \cite{Passare:Rullgard:Spine}. Let $\Omega$ be a convex open set in $\R^n$ and let $f \in \C[\mathbf{z}^{\pm 1}]$ be a Laurent polynomial that is defined on $\Log^{-1}|\Omega|$. The \struc{\textit{Ronkin function}} $R_f$ is defined by the integral
\begin{eqnarray*}
 \struc{R_f}: \Omega \to \R, \quad \mathbf{x} \mapsto \frac{1}{(2\pi i)^n} \int_{\Log^{-1}|\mathbf{x}|} \frac{\log|f(\mathbf{z})| dz_1 \ldots dz_n}{z_1 \ldots z_n}.
\end{eqnarray*}

The next theorem collects some important properties of $R_f$.

\begin{thm}(Ronkin \cite{Ronkin} / Passare, Rullg{\aa}rd \cite[Section 1]{Passare:Rullgard:Spine})
Let $f\in \C[\mathbf{z}^{\pm 1}]$ be a holomorphic function. Then $R_f$ is a convex function. If $U \subseteq \Omega$ is a connected open set, then the restriction of $R_f$ to $U$
is affine linear if and only if $U\cap \cA(f)=\emptyset$. If $\mathbf{x}$ is
in the complement of $\cA(f)$, then the gradient of $R_f(\mathbf{x})$ equals the order of $\mathbf{x}$.
\label{Thm:RonkinProperties}
\end{thm}

Given a Laurent polynomial $f$ and a point $\alp$ in its support set $A$, we have seen in Section \ref{SubSec:OrderMap} that the corresponding component $E_\alp(f)$ of the complement of $\cA(f)$ is non-empty if and only if $\alp \in \im(\ord(f))$. For every $\alp \in \im(\ord(f))$ one defines the \struc{\textit{Ronkin coefficient}} of $\alp$ by
\begin{eqnarray}
	\struc{r_\alp}		& := & R_f(\mathbf{x}) - \langle \alp , \mathbf{x} \rangle \ \text{ for every } \mathbf{x} \in E_{\alp}(f),\label{Equ:RonkinCoefficient}
\end{eqnarray}
which is well-defined due to Theorem \ref{Thm:RonkinProperties}. The Ronkin coefficients give rise to the following tropical polynomial
\begin{eqnarray}
	\struc{\text{SpineT}(f)} 	& := &  \bigoplus_{\alp \, \in \, \im(\ord(f))} r_\alp \oplus \mathbf{x}^{\alp}. \label{Equ:SpineTrop}
\end{eqnarray}
Finally, the \struc{\textit{spine}} of $f$ is the tropical hypersurface $\cS(f)$ given by $\text{SpineT}(f)$, i.e., 
\begin{eqnarray*}
	\struc{\cS(f)}		& := & \cT(\text{SpineT}(f)).
\end{eqnarray*}
It was shown by Passare and Rullg{\aa}rd \cite{Passare:Rullgard:Spine} that the spine of a Laurent polynomial $f \in \C[\mathbf{z}^{\pm 1}]$ is a deformation retract of $\cA(f)$. Note that retracting an amoeba to its spine does not change the logarithmic limit set of the amoeba. For additional background about the spine and the relation between amoebas and tropicalizations we recommend \cite{deWolff:AmoebaTropicalizationSurvey, Passare:Tsikh:Survey}.

\subsection{Mixed Volumes and Bernstein's Theorem}

Recall that the \struc{\textit{Minkowski sum}} of  $n$-polytopes $P$ and $Q$ is defined as:
\begin{eqnarray*}
 \struc{P + Q} & := & \{\mathbf{p} + \mathbf{q} \ : \ \mathbf{p} \in P, \ \mathbf{q} \in Q\};
\end{eqnarray*}
see \cite[p. 28]{Ziegler:Book}. Following \cite[p. 205]{Gelfand:Kapranov:Zelevinsky}, we provide the following definition of \struc{\textit{mixed volumes}}. Let $P_1,\ldots,P_n$ be polytopes in $\R^n$ such that $\dim (P_1+\cdots+P_n)=n$. For $\lam \in \R$ let  $\struc{\lam P_i} := \{\lam \mathbf{p} \ : \ \mathbf{p} \in P\}$. Given a translation-invariant volume form $\Vol$ on $\R^n$, the expression
\begin{eqnarray}
 \struc{\Vol(\lam_1 P_1 + \cdots + \lam_n P_n)} \label{Equ:MixedVolume}
\end{eqnarray}
is a homogeneous polynomial in $\lam_1,\ldots,\lam_n$ of degree $n$.

\begin{definition}
The \struc{\emph{mixed volume} $\Vol(P_1,\ldots,P_n)$} is the coefficient of the monomial $\lam_1 \cdots \lam_n$ in the polynomial \eqref{Equ:MixedVolume}.
More explicitly, we have
\begin{eqnarray*}
 \Vol(P_1,\ldots,P_n) & := & \frac{1}{n!} \sum_{j = 1}^n (-1)^{n - j} \sum_{1 \leq i_1 < \cdots < i_j \leq n} \Vol(P_{i_1}+\cdots +P_{i_j}).
\end{eqnarray*}
\end{definition}

We denote by  \struc{$\MV$} the normalized volume form that is induced by the lattice $\Z^n$ and for which every standard lattice simplex has volume 1; see \cite[Chapter 5, Section 3D]{Gelfand:Kapranov:Zelevinsky}. The classical Bernstein Theorem states the following; see \cite{Bernstein} and \cite[Theorem 2.8., p. 206]{Gelfand:Kapranov:Zelevinsky}:

\begin{thm}[Bernstein's Theorem]
 Let $A_1,\ldots,A_n \subset \Z^n$ be finite sets such that their union generates $\Z^n$ as an affine lattice. Let $P_i \subseteq \R^n$ be the convex hull of $A_i$. 
Then there exists a dense Zariski open subset $U \subseteq \C^{A_1} \times \cdots \times \C^{A_n}$ such that for any $(f_1,\ldots,f_n) \in U$, the number of solutions of the system of equations $f_1(\mathbf{z}) = \cdots = f_n(\mathbf{z}) = 0$ in $(\C^*)^n$ equals the mixed volume $\MV(P_1,\ldots,P_n)$.
\end{thm}

Let $P_1,\ldots,P_n$ be $n$-dimensional lattice polytopes and let $P := P_1 + \cdots + P_n$. A sum $\struc{C} := F_1 + \cdots + F_n$, where  $F_i \subseteq P_i$ is a face ($1\leq i\leq n$), is called a \struc{\textit{cell}} of $P$. A \struc{\textit{subdivision}} of $P$ is a collection $\struc{\Gamma} := \{C_1,\ldots,C_m\}$ of cells such that each cell is of full dimension, the intersection of two cells is a face of both and the union of all cells covers $P$. A subdivision $\Gamma$ is called \struc{\textit{mixed}} if for each cell $C=F_1+\cdots+F_n\in \Gamma$ one has $n=\dim F_1+\cdots+\dim F_n$. 
A cell $C$ is called \struc{\textit{mixed}} if every $P_i$ contributes with a face of dimension at least $1$ (and hence, in our setting, of dimension exactly $1$) to $C$. For further information and the following statement see \cite{Steffens:Theobald:TropicalIntersection}; see also \cite{Huber:Sturmfels} and \cite[Section 4.6.]{Maclagan:Sturmfels}
\begin{lemma}
Let $P_1,\ldots,P_n$ be $n$-dimensional lattice polytopes and let $\Gamma$ be a mixed subdivision of $ P_1 + \cdots + P_n$. Then we have
\begin{eqnarray*}
 \MV(P_1,\ldots,P_n) & = & \sum_{C \text{ mixed cell in } \Gamma} \MV(C).
\end{eqnarray*}
\label{Lem:MixedCells}
\end{lemma}

An intersection of $k$ tropical hypersurfaces in $\R^n$ is called \struc{\textit{proper}} if it has codimension $k$. If $k$ tropical hypersurfaces $\mathcal{T}_1,\ldots,\mathcal{T}_k$ intersect non-properly, one can consider their \struc{\textit{stable}} intersection, which is defined as follows:  
 It is defined as the limit, for $\eps \to 0$, of an $\eps$-perturbation of the original intersection. The crucial point here is that every $\eps$-perturbation intersects transversely and also every stable intersection is transverse and in particular proper; see \cite[Section 3.1.]{Steffens:Theobald:TropicalIntersection} and \cite[Section 3.6.]{Maclagan:Sturmfels} for further details. We denote the stable intersection of $\mathcal{T}_1,\ldots,\mathcal{T}_k$ by $\mathcal{T}_1\cap_{\mathrm{st}}\cdots \cap_{\mathrm{st}}\mathcal{T}_k$. 

We can now state the tropical analogue of Bernstein's Theorem. We provide a short version here, which is sufficient for our needs. For more information about the tropical Bernstein theorem including the detailed version see \cite[Theorem 4.6.9., p. 196]{Maclagan:Sturmfels}.

\begin{thm}[Tropical Bernstein Theorem]
Let $\cT(h_1),\ldots,\cT(h_n) \in \R^n$ be generic tropical hypersurfaces which are dual to regular subdivisions of $\New(h_1),\ldots,\New(h_n) \subseteq \R^n$. The multiplicity of each point $\mathbf{w}$ in the stable intersection $\cT(h_1) \cap \cdots \cap \cT(h_n)$ equals the mixed volume $\MV(C)$, where $C$ is the mixed cell in the subdivision of $\New(h_1) + \cdots + \New(h_n)$ induced by $\cT(h_1) \cap \cdots \cap \cT(h_n)$ corresponding to $\mathbf{w}$. 
\label{Thm:TropicalBernsteinTheorem}
\end{thm}

\section{Combinatorics of Intersections of Amoebas}
\label{Sec:Combinatorics}
The aim of this section is to study basic combinatorial properties of intersections of amoebas of hypersurfaces. 

We start by fixing some notation that we use during this and the next section. In what follows, we always assume that $n\geq 2$. We fix support sets $A_1,\ldots,A_n\subseteq \Z^n$ and let $\cF:=\{f_1,\ldots,f_n\}\subseteq\C[\mathbf{z}^{\pm 1}]$ be a collection of Laurent polynomials such that $f_i\in(\C^*)^{A_i}$ for all $1\leq i\leq n$. Throughout this article, we call a collection $\cF:=\{f_1,\ldots,f_n\}\subseteq\C[\mathbf{z}^{\pm 1}]$ of Laurent polynomials \struc{\textit{generic}} in the space $(\C^*)^A := (\C^*)^{A_1} \times \cdots \times (\C^*)^{A_n}$ if the following conditions hold:
\begin{itemize}
\item[(1)] every $f_j \in \cF$ is irreducible,
\item[(2)] for every $1\leq k\leq n$ and every $k$-element subset $\{f_{i_1},\ldots, f_{i_k}\}\subseteq \cF$ the intersection $\bigcap_{\ell=1}^k \partial(\cA(f_{i_\ell}))$ has codimension $k$,
\item[(3)] for every $p \in \bigcap_{j = 1}^n \cA(f_j)$ and every $1\leq \ell \leq n$ such that $p\in \partial(\cA(f_\ell))$ there exists a unique component of the complement of $\cA(f_\ell)$ containing $p$ in its closure. 
\end{itemize}
Moreover, a single Laurent polynomial $f$ is \struc{\emph{generic}}, if it satisfies condition (1). Note that the term ``generic'' is justified as follows: (1) is obviously a Zariski open condition.  
 For conditions (2) and (3) it is less obvious why it corresponds to an open set; we give a more detailed explanation in Appendix \ref{Sec:AppendixGenericity}. In Appendix \ref{Sec:AppendixDimension} we give a short description how the term ``dimension'' is defined here.

We remark that this kind of genericity implies several consequences, which are crucial in some of the later proofs. In particular, the boundaries of all amoebas in $\cF$ have a non-trivial intersection. Thus, the boundaries of any two amoebas intersect in codimension $2$.

In what follows, we consider a generic collection  $\cF := \{ f_1,\ldots,f_n \}\subseteq \C[{\mathbf z}^{\pm 1}]$ of Laurent polynomials 
 and we are interested in the intersection $\cI(\cF):=\bigcap_{j = 1}^n \cA(f_j)$. An initial example shows that $\cI(\cF)$ is disconnected in general.

\begin{example}
Let $\cF:=\{f_1,f_2\}\subseteq \C[z_1^{\pm 1},z_2^{\pm 1}]$, where $f_1(z_1,z_2) := 2z_1 + z_2 + 1$ and $f_2(z_1,z_2) := z_1^2z_2 + z_1z_2^2 + 5 z_1z_2 + 1$. One can see from Figure \ref{Fig:AmoebaIntersection1} that $\cA(f_1) \cap \cA(f_2)$ consists of two disjoint connected components.
\end{example}

We investigate basic properties of the connected components of $\cI(\cF)$. In order to do so, we need the following lemma concerning strictly convex sets.

\begin{lemma}\label{lem:convex}
Let $n\geq 2$ and let $B\subseteq \R^n$ be an $n$-dimensional closed ball. Let $s\geq 2$ and let $L_1,\ldots,L_s\subseteq \R^n$ be $n$-dimensional strictly convex sets satisfying the following conditions:
\begin{itemize}
\item[(a)] $\bigcap_{i=1}^sL_i\cap B^{\circ}$ is non-empty of dimension smaller than or equal to $n-1$, where $B^{\circ}$ denotes the interior of $B$. 
\item[(b)] $B\subseteq \bigcup_{i=1}^s L_i$.
\end{itemize}
Then $s\geq n+1$.
\end{lemma}

\begin{proof}
We first show that condition (a) implies that $D:=\bigcap_{i=1}^sL_i\cap B^{\circ}$ equals a single point. To see this, one first observes that $D$ has to lie in the boundary of each $L_i$. Otherwise, it is easy to see that, since $B$ and $L_i$ are strictly convex, $D$ has to be $n$-dimensional (note that this statement is not necessarily true if the sets are just convex). If $p,q\in D$, then, by convexity of $D$, we also know that the line segment $[p,q]$ between $p$ and $q$ has to lie in $D$. But as $B$ and $L_i$ are strictly convex, it follows that $[p,q]$ belongs to the interior of $L_i$, which implies that $\dim D=n$, a contradiction.

We now prove the statement by induction on $n$. For $n=2$, assume, by contradiction, that there exist two strictly convex sets $L_1$, $L_2$ satisfying conditions $(a)$ and $(b)$. Then, by the above statement, it follows that $L_1\cap L_2\cap B$ equals a single point $v$. Since $L_1$ and $L_2$ are both strictly convex, there exists a line $L$ passing through $v$ such that $L_i\cap L=\{p\}$ for $1\leq i\leq 2$ and $L_1$ and $L_2$ lie on different sides of $L$. But then $\emptyset\neq (L\setminus\{p\})\cap B\not\subseteq L_1\cup L_2 $ which contradicts $(b)$. 

Now assume, $n\geq 3$. Let $s\geq 2$ and  $L_1,\ldots,L_s$ strictly convex such that $(a)$ and $(b)$ hold. As for $n=2$, strict convexity implies that $\bigcap_{i=1}^sL_i\cap B^{\circ}$ consists of a single point $v$. Let $H$ be a hyperplane in $\R^n$ containing $v$. Then $H\cap L_1,\ldots,H\cap L_s$ yields a covering of the $(n-1)$-dimensional ball $B\cap H$ with strictly convex sets such that $(a)$ and $(b)$ hold. By induction we get $s\geq n$.  (b) implies that there exists an $i$ such that $H\cap L_i$ is of dimension $n-1$. Moreover, as $L_i$ is strictly convex, there exists a hyperplane $G\subseteq \R^n$ such that $L_i$ is contained in one halfspace of $G$ and intersects $G$ only in $v$. Applying induction to the $(n-1)$-dimensional ball $B\cap G$, we know that there are at least $n$ strictly convex sets necessary to cover $B\cap G$. Since $L_i$ does intersect $B\cap G$ only in $v$, it does not contribute to the cover of $B\cap G$ and $L_i\cap G$ could hence be removed.  We conclude, that there exist at least $n$ strictly convex sets in the collection $L_1,\ldots,L_s$ different from $L_i$. Therefore, we have $s\geq n+1$.
\end{proof}

The next theorem states some basic properties of the connected components of  $\cI(\cF)$.

\begin{thm}
Let $n\geq 2$ and let $\cF:=\{f_1,\ldots, f_n \}\subseteq  \C[\mathbf{z}^{\pm 1}]$ be generic. Then every connected component of  $\cI(\cF)$ is a closed, $n$-dimensional set. If the spines $\cS(f_1),\ldots,\cS(f_n)$ intersect properly, then every connected component is compact. 
\label{Thm:Compact}
\end{thm}

\begin{proof}
Consider a connected component $K$ of  $\cI(\cF)$. First, we show that $K$ is bounded if the spines $\cS(f_1),\ldots,\cS(f_n)$ intersect properly.
Recall the definition of $\cA_r(f)$ from \eqref{Equ:ScaledAmoeba}. Assume by contradiction that $K$ is unbounded. It follows that there exists $s\gg 0$ such that for all $r \in \R_{\geq s}$ we have $\bigcap_{j = 1}^n \cA_r(f_j) \neq \emptyset$. Thus, $\lim_{r \to \infty} \bigcap_{j = 1}^n \cA_r(f_j) \neq \emptyset$ and hence $\bigcap_{j = 1}^{n} \cA_{\infty}(f_j) \neq \emptyset$. This means that the intersection of the spines $\cS(f_1),\ldots ,\cS(f_n)$ has codimension smaller than $n$. This contradicts the assumption that the spines $\cS(f_1),\ldots,\cS(f_n)$ intersect properly.

Second, we show that $K$ is closed. But this is immediate since $\bigcap_{j = 1}^n \cA(f_j)$ is the intersection of closed sets. 

We now show that $K$ is $n$-dimensional. 
Given a subset $J\subseteq [n]:=\{1,\ldots,n\}$, we set $\cF_J:=\{f_j~:~j \in J\}$ and $\cI(\cF_J):=\bigcap_{j\in J}\cA(f_j)$. 
We show by induction on $\#J$ that for all $J\subseteq [n]$ any connected component of $\cI(\cF_J)$ has dimension $n$.\\
Since $n \geq 2$, the statement is clear for $\#J=1$. \\
Let $\#J=m\geq 2$. Without loss of generality let $J=[m]$. By contradiction, assume that $K$ is a connected component of $\cI(\cF_{[m]})$ such that $\dim K<n$. 
First, we show that $K$ has to be contained in the boundary of $\cA(f_j)$ for $1\leq j \leq m$, in this case. Suppose without loss of generality that $K\subseteq \cA(f_m)\setminus \partial (\cA(f_m))$. Consider a point  $p$ in the relative interior of $K$. Since $p$ lies in the interior of $\cA(f_m)$, there exists a small $n$-dimensional neighborhood $U$ of $p$, such that $K\cap U= \cI(\cF_{[m-1]})\cap U$. By induction, it follows that $\cI(\cF_{[m-1]})\cap U$ is $n$-dimensional, hence a contradiction to the assumption that $\dim(K) < n$.\\
Being $K$ contained in $\partial(\cA(f_j))$ for all $1\leq j\leq m$, means that $K$ lies in the intersection of the boundaries of specific components $E_{\alpha(j)}(f_j)$ of the complement of $\cA(f_j)$ for $1\leq j\leq m$. 

Hence, if $q$ is a point in the relative interior of $K$, then there exists a small $n$-dimensional neighborhood $V$ of $q$ such that $\left(\bigcup_{j=1}^m \overline{E_{\alpha(j)}(f_j)}\right)\cap V=V$. 
As $2\leq m\leq n<n+1$, this yields a contradiction to Lemma \ref{lem:convex}.
\end{proof}

Note that in case that the spines $\cS(f_1),\ldots,\cS(f_n)$ do not intersect properly, one can easily compactify the set  $\cI(\cF)$ by considering the \struc{\textit{compactified amoebas}} instead of the usual amoebas $\cA(f_1),\ldots,\cA(f_n)$. Compactified amoebas are obtained via the toric moment map and are well-known objects in amoeba theory. They were already considered by Gelfand, Kapranov, and Zelevinsky; see \cite[p. 198 et seq.]{Gelfand:Kapranov:Zelevinsky} and \cite{Mikhalkin:Survey}. For convenience of the reader we omit this additional technicality and, in what follows, we assume compactness of $\cI(\cF)$, where it is needed.

In order to obtain information about the connected components of the intersection $\cI(\cF)$ of $\cA(f_1),\ldots,\cA(f_n)$, it is important to understand the boundaries of those components. Clearly, they can be described by means of intersections of boundaries of specific components of the complements of subcollections of $\cA(f_1),\ldots,\cA(f_n)$.  Condition (2) of the definition of genericity implies that for a generic collection $\cF := \{f_1,\ldots,f_n\} \subset \C[\mathbf{z}^{\pm 1}]$ of $n$ Laurent polynomials the intersection $\bigcap_{j = 1}^k \partial (\cA(f_j) )$ is zero dimensional and of finite cardinality. 
This motivates the following definition. 

\begin{definition}
 Let $\cF:=\{f_1,\ldots f_n\}\subseteq \C[\mathbf{z}^{\pm 1}]$ be a generic collection of Laurent polynomials.
 \begin{itemize}
  \item[(a)] $\bigcap_{j = 1}^n \partial(\cA(f_j))$ is called the set of \struc{\emph{vertices}} of $\cI(\cF)$, denoted by \struc{$V(\cF)$}.
  \item[(b)] For a connected component $K$ of $\cI(\cF)$ we call $\struc{V(K)}:= K \cap V(\cF)$ the set of \struc{\emph{vertices}} of $K$.
 \end{itemize}
 \label{Def:VerticesOfComponents}
\end{definition}

By Definition \ref{Def:VerticesOfComponents} vertices of a connected component $K$ of $\cI(\cF)$ lie on the boundary of $K$. More generally, we can decompose the boundary of $K$ as the union of disjoint pieces, each of which is contained in the intersection of finitely many $\partial \ovl{E_{\alpha(j)}(f_j)}$. This motivates the following definition of $k$-\textit{faces} of $K$. 

\begin{definition}\label{def:Face}
  Let $\cF:=\{f_1,\ldots f_n\}\subseteq \C[\mathbf{z}^{\pm 1}]$ be a generic collection of Laurent polynomials. Let $K$ be a connected component of  $\cI(\cF)$. Let $0\leq k\leq n-1$.
  \begin{itemize}
   \item[(a)] A non-empty and connected subset $\struc{F}\subsetneq K$ is called a \struc{$k$-\emph{face}} of $K$, if there exist unique  $E_{\alpha(1)}(f_{j_1}),\ldots,E_{\alpha(n-k)}(f_{j_{n-k}})$ such that  $\mathbf{x} \in \bigcap_{s = 1}^{n-k} \partial \ovl{E_{\alpha(s)}(f_{j_s})}$ for all $\mathbf{x} \in F$.
   \item[(b)] An $(n-1)$-dimensional face of $K$ is called a \struc{\emph{facet}} of $K$.
  \end{itemize}
\end{definition}
We remark that Definition \ref{Def:VerticesOfComponents} combined with our definition of genericity implies the existence of a \textit{face lattice} for every connected component $K$ of $\cI(\cF)$, where, as usual, faces are ordered by inclusion.

Note that a priori the definition of a face does not exclude that a single amoeba contributes with multiple components of its complement to an intersection that describes a specific face. The following lemma, however, shows that this case can never occur.

\begin{lemma}
 Let $\cF:=\{f_1,\ldots f_n\}\subseteq \C[\mathbf{z}^{\pm 1}]$ be generic. Let $K$ be a connected component of $\cI(\cF)$. Let $F$ be a $k$-face of $K$ that is  given by $E_{\alpha(1)}(f_{j_1}),\ldots,E_{\alpha(n-k)}(f_{j_{n-k}})$. Then all the $f_{j_s}$ are distinct.
\label{Lem:FaceConstruction}
\end{lemma}

\begin{proof}
It suffices to observe that condition (3) in the definition of genericity implies that for every $f_i$ and different non-empty components of the complement $E_\alp(f_i)$ and $E_{\beta}(f_i)$ of $\cA(f_i)$, one has  $\ovl{E_{\alp}(f_i)} \cap \ovl{E_{\beta}(f_i)} = \emptyset$.
\end{proof}
Note that property (3) of the definition of genericity implies that every vertex lies in a unique component of the complement of every single amoeba.  Hence, as a consequence of the previous lemma, it follows that $0$-faces as defined in Definition \ref{def:Face} coincide with vertices as defined in Definition \ref{Def:VerticesOfComponents}. 

\
Given a connected component $K$ of $\cI(\cF)$ with set of vertices $V(K)$, we define the polytope \struc{$P_K$} as the convex hull of $V(K)$. We call all polytopes arising in this way \struc{\textit{intersection polytopes}} of $\cF$. We remark that though the vertex set of an intersection polytope $P_K$ is clearly contained in $V(K)$, it does not have to coincide with $V(K)$; see Figure \ref{Fig:VerticesComponentsVersusConvexHulls}. In the following, we use \struc{$V(P_K)$} to denote the vertex set of $P_K$.

\begin{figure}
 \ifpictures
\includegraphics[width=0.4\linewidth]{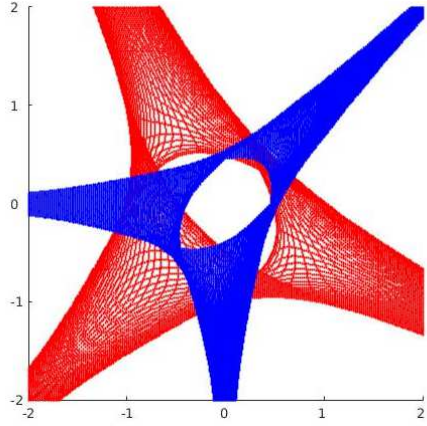}
\fi
\caption{The intersection of $f_1(z_1,z_2) := z_1^2z_2^2 + z_1z_2 + z_1 + z_2$ and $f_2(z_1,z_2) := z_1^3 + z_2^3 + 2z_1z_2 + 1$. None of the vertices of $\cI(\{f_1,f_2\})$ given by the intersection of the boundaries of the two bounded components of the complement of $\cA(f_1)$ and $\cA(f_2)$ are vertices of the corresponding intersection polytopes.}
\label{Fig:VerticesComponentsVersusConvexHulls}
\end{figure}

The following theorem shows that the vertex set $V(P_K)$ of $P_K$ coincides with the set of extreme points of $\conv(K)$.

\begin{thm}
Let $\cF:=\{f_1,\ldots,f_n\}\subseteq \C[\mathbf{z}^{\pm 1}]$ be generic and let $K$ be a connected component of $\cI(\cF)$. Then the vertex set $V(P_K)$ of $P_K$ is given by the set of extreme points of $\conv(K)$. In particular, we have $\conv(K) = P_K$.
\label{Thm:ExtremePoints}
\end{thm}

\begin{proof}
Let $E(K)$ be the set of extreme points of $\conv(K)$. First, we show that $E(K)\subseteq V(K)$.
Let $p\in \conv(K)$ be an extreme point. Then $p$ has to lie on the boundary of $K$ and hence there exists a face of $K$ containing $p$. Let $F$ be the smallest
such face of $K$ with respect to inclusion and dimension and let $k=\dim F$. Without loss of generality (after renumbering), suppose that $F$ can be described by $\bigcap_{j=1}^k \partial\ovl{E_{\alp(j)}(f_j)}$. If $k=n$, then by definition of vertices of $K$, it already follows that $p$ itself is a vertex of $K$. So, assume $k<n$. For $\eps > 0$, we denote by $B_{\eps}(p) \subseteq \R^n$ the closed $n$-ball with radius $\eps$ around $p$. Consider the intersection $B_{\eps}(p) \cap K$. Since $K$ is $n$-dimensional, we have $B_{\eps}(p)\cap (K\setminus \partial K)\neq \emptyset$ and, if $\eps$ is sufficiently small, then 
\begin{eqnarray*}
B_{\eps}(p) \cap \partial K \ = \ B_{\eps}(p) \cap \partial K \cap \left (\bigcap_{j = 1}^k \partial \cA(f_j)\right) & = & B_{\eps}(p) \cap \left (\bigcap_{j = 1}^k \partial \ovl{E_{\alp(j)}(f_j)}\right ).
\end{eqnarray*}
Our genericity condition (2) implies that the intersection $B_{\eps}(p) \cap \left (\bigcap_{j = 1}^k \partial \ovl{E_{\alp(j)}(f_j)}\right )$ is of dimension $n-k\geq 1$ and contains $p$ in its interior. 
 By convexity of the components of the complement of an amoeba, there exist points $p_1,\ldots,p_{n-k},p_{n-k+1}\in B_{\eps}(p) \cap \partial K$ and $q\in B_{\eps}(p)\cap(K\setminus \partial K)$, spanning an $(n-k+1)$-simplex $\Gamma$ that contains $p$ in its interior. If $\eps$ is sufficiently small, then one has $\Gamma\subseteq \conv(K)$. Hence, $p$ lies in the interior of $K$, which is a contradiction since $p$ was chosen to be an extreme points. Thus, $k<n$ cannot happen and this implies $E(K)\subseteq V(K)$. 
 
 We conclude the ``In particular''-statement. By the Krein-Milman Theorem \cite{Krein:Milman}, $\conv(K)$ is the convex hull of its extreme points and we hence obtain
\begin{equation*}
 \conv(K) \ = \ \conv (E(K)) \ \subseteq \ \conv (V(K)) \ = \ P_K.
\end{equation*}
The other inclusion $P_K\subseteq \conv(K)$ follows directly from $V(K)\subseteq K$.

It remains to show that every extreme point of $\conv (V(K))$ is indeed a vertex of $P_K$, i.e., $E(K)\subseteq V(P_K)$. Since $\conv(K)=P_K=\conv( V(K))$ it follows that $\conv(K)$ is a polytope. As such its set of extreme points $E(K)$ and its set of vertices $V(P_K)$ coincide, which shows the claim. 
\end{proof}

We recall that an $n$-dimensional polytope $P$ is \struc{\emph{simple}} if every vertex of $P$ lies in exactly $n$ facets.  
For intersection polytopes the following statement holds.

\begin{prop}
Let $\cF:=\{f_1,\ldots,f_n\}\subseteq\C[\mathbf{z}^{\pm1}]$ be a generic collection of Laurent polynomials and let $K$ be a connected component of $\cI(\cF)$. Then, the corresponding intersection polytope $P_K$ is simple. 
\end{prop}

\begin{proof}
 Let $p$ be a vertex of $P_K$. Since $P_K$ is $n$-dimensional, $p$ has to lie in at least $n$ facets. On the other hand, as $p$ is a vertex of $P_K$, there exist components $E_{\alpha(1)}(f_1),\ldots,E_{\alpha(n)}(f_n)$ of the complements of the corresponding amoebas such that $p$ lies in the intersection of their boundaries. Since every facet containing $p$ is determined by one of those components $E_{\alp(j)}(f_j)$ and since, vice versa, every such component of the complement determines no more than one facet, we conclude that $p$ lies in at most $n$ facets. The claim follows.
\end{proof}

Let $\cF:=\{f_1,\ldots,f_n\}\subseteq \C[\mathbf{z}^{\pm 1}]$. Since a spine is a deformation retract of its corresponding amoeba, it is natural to ask if the different connected components of the intersection $\cI(\cF)$ can already be detected from the intersection of the spines of $\cA(f_1),\ldots,\cA(f_n)$. On the one hand, every intersection point of the spines has to lie in a connected component of $\cI(\cF)$. On the other hand, it is not clear a priori whether every connected component of $\cI(\cF)$ contains a common point of intersection of the spines of $\cA(f_1),\ldots,\cA(f_n)$. To provide an answer to the latter question, we need the following preparatory lemma.

\begin{lemma}
Let $\cF := \{f_1,\ldots,f_n\}\subseteq  \C[\mathbf{z}^{\pm 1}]$ be a generic collection of Laurent polynomials and let $K$ be a connected component of $\cI(\cF)$. For every $1\leq \ell\leq n$, there exist $\alpha$, $\beta\in\New(f_\ell)\cap \Z^n$, with $\alpha\neq \beta$ such that 
\begin{equation*}
\ovl{E_{\alpha}(f_\ell)} \cap K \neq \emptyset \quad \mbox{and}\quad \ovl{E_{\beta}(f_\ell)} \cap K \neq \emptyset.
\end{equation*}
\label{Lem:IntersectionTwoComplementComponents}
\end{lemma}

\begin{proof}
Assume by contradiction that there exists an $f_\ell$ such that $\ovl{E_{\alpha}(f_\ell)} \cap K\neq \emptyset$ only for a single $\alpha\in \New(f)\cap \Z^n$. Since, by Lemma \ref{Lem:FaceConstruction}, every amoeba $\cA(f_1),\ldots,\cA(f_n)$ contributes to every vertex of $K$, it follows that all vertices of $K$ lie on $\partial\ovl{E_{\alpha}(f_\ell)}$. By convexity of $E_{\alpha}(f_\ell)$ we conclude that $P_K\subseteq \ovl{E_{\alpha}(f_\ell)}$ and thus also $K\subseteq \ovl{E_{\alpha}(f_\ell)}$. As $K\subseteq \cA(f_\ell)$ by definition of $K$, we infer that $K\subseteq \partial \ovl{E_{\alpha}(f_\ell)}$, which implies that $\dim K\leq n-1$. This contradicts Lemma \ref{Thm:Compact} (b) and hence the claim follows. 
\end{proof}

The next theorem shows that for a generic collection of Laurent polynomials $\cF:=\{f_1,\ldots,f_n\}\subseteq \C[\mathbf{z}^{\pm 1}]$ each connected component of $\cI(\cF)$ contains an intersection point of the spines. In particular, if the spines of $\cA(f_1),\ldots,\cA(f_n)$ intersect properly, then their intersection consists only of finitely many points, whose number provides an upper bound for the number of connected components of $\cI(\cF)$

\begin{thm}
Let $\cF:=\{f_1,\ldots,f_n\}\subseteq\C[\mathbf{z}^{\pm 1}]$ be a generic collection of Laurent polynomials and let $K$ be a connected component of $\cI(\cF)$. Then:
\begin{equation*}
K \cap \bigcap_{j = 1}^n \cS(f_j) \ \neq \ \emptyset.
\end{equation*}
Moreover, 
\begin{equation*}
K\cap \cS_(f_1)\cap_{\mathrm{st}}\cdots\cap_{\mathrm{st}}\cS(f_n) \ \neq \ \emptyset.
\end{equation*}
If the spines $\cS(f_1),\ldots,\cS(f_n)$ intersect properly, then the number of connected components of $\cI(\cF)$ is at most $\#\bigcap_{j=1}^n\cS(f_j)$. 
\label{Thm:SpineIntersectsComponents}
\end{thm}

\begin{proof}
Let $K$ be a connected component of $\cI(\cF)$. We will show that for every $J\subseteq [n]$ the intersection $\bigcap_{\ell\in J}\cS(f_\ell)\cap K$ is non-empty and of dimension at least $n-\#J$. 
We proceed by induction on $\#J$.  \\
Let $\#J=1$, i.e., $J:=\{\ell\}$ for some $1\leq \ell\leq n$. By Lemma \ref{Lem:IntersectionTwoComplementComponents}, we know that $K$ intersects the closure of at least two different components $E_{\alpha}(f_\ell),E_{\beta}(f_\ell)$ of the complement of $\cA(f_\ell)$ non-trivially. Let $p$ and $q$ be points in $\ovl{E_{\alpha}(f_\ell)}$ and $\ovl{E_{\beta}(f_\ell)}$, respectively. Since $K$ is connected, there exists a path $\gamma$ from $p$ to $q$ inside $K$. Since $\cS(f_\ell)$ is a deformation retract of $\cA(f_\ell)$ (see Section \ref{SubSec:Tropical}), it follows that $\cS(f_\ell)$ intersects the interior of $\gamma$. Moreover, as $K$ is of dimension $n$ and as $\cS(f_\ell)$ is a tropical $(n-1)$-dimensional hypersurface lying in the interior of $\cA(f_\ell)$, we further conclude that $K\cap\cS(f_\ell)$ is of dimension $n-1$.\\
Now, suppose $\#J=k\geq 2$. Without loss of generality, assume $J=[k]$. It follows by induction that $Q:=\bigcap_{\ell=1}^{k-1}\cS(f_\ell)\cap K$ is non-empty and of dimension at least $n-(k-1)$. Furthermore, assume that $Q$ is connected (otherwise consider a connected component in the sequel). Being part of the intersection of $k-1$ tropical hypersurfaces,  $Q$ has to intersect the boundary of $\partial(K)$. Moreover, as any spine lies in the interior of its amoeba, $Q$ cannot intersect facets of $K$ defined by $\cA(f_1),\ldots,\cA(f_{k-1})$. Since $\bigcap_{\ell=1}^{k-1}\cS(f_\ell)$ is unbounded but $Q$ itself is bounded, it follows that  $Q$ has to intersect  the boundary of every $\cA(f_k),\ldots,\cA(f_n)$. Moreover, using Lemma \ref{Lem:IntersectionTwoComplementComponents}, we can further conclude that $Q$ intersects the boundaries of two different non-empty components, $E_{\alpha}(f_k)$ and $E_\beta(f_k)$ of the complement of $\cA(f_k)$. Let $p\in Q\cap \ovl{E_\alpha(f_k)}$ and $q\in Q\cap \ovl{E_\beta(f_k)}$. Since $Q$ is connected, there exists a path $\gamma$ connecting $p$ and $q$ that lies inside of $Q$. As in the case $\#J=1$ we can conclude that $\gamma$ intersects $\cS(f_k)$ and in particular $Q\cap \cS(f_k)\neq \emptyset$. 
Since the latter is the intersection of $k$ generic tropical hypersurfaces with the $n$-dimensional closed set $K$, we infer that $\bigcap_{\ell=1}^{k}\cS(f_\ell)\cap K$ is of dimension at least $(n-k)$. This completes the proof of the first part.  
If the intersection is proper, then the proof of the first part shows that the dimension of $\bigcap_{\ell=1}^n\cS(f_\ell)\cap K$ is indeed $0$-dimensional. Hence, any connected component of $\cI(\cF)$ contains at least one point of $\bigcap_{j = 1}^n \cS(f_j)$ and the upper bound for the number of connected components follows.

It remains to show the ``Moreover''-part. Let $K$ be a connected component of $\cI(\cF)$. Consider a collection of Laurent polynomials $\cF^\eps=\{f_1^{\eps},\ldots,f_n^\eps\}$ obtained by perturbing the coefficients of $f_1,\ldots,f_n$ by $\eps$. If $\eps$ is sufficiently small, then the connected components of $\cI(\cF)$ and $\cI(\cF^\eps)$ are in $1$ to $1$-correspondence. Let $K^\eps$ be the component corresponding to $K$. By the first part we know that $K^\eps\bigcap_{\ell=1}^n\cS(f_\ell^\eps)\neq \emptyset$ and, choosing $\eps$ sufficiently small, we can even assume that $K\bigcap_{\ell=1}^n\cS(f_\ell^\eps)\neq \emptyset$. As $K$ is closed, it follows that 
$$
\lim_{\eps\to\infty}\left(K\bigcap_{\ell=1}^n\cS(f_\ell^\eps)\right)\neq \emptyset.$$
 The claim follows.
\end{proof}

As an almost immediate consequence of the previous theorem we obtain the Amoeba Bernstein Theorem.

\begin{Bernstein}
Let $\cF:=\{f_1,\ldots,f_n\}\subseteq \C[\mathbf{z}^{\pm 1}]$ be a generic collection of Laurent polynomials and let $c$ be the number of connected components of $\cI(\cF)$. Then $c$ is bounded from above by the number of mixed cells in the subdivision of $\New(f_1) + \cdots + \New(f_n)$, which is induced by the subdivisions of $\New(f_1),\ldots,\New(f_n)$ that are dual to the spines $\cS(f_i),\ldots,\cS(f_n)$. In particular, $c$ is bounded from above by the mixed volume $\MV(\New(f_1),\ldots,\New(f_n))$.

Moreover, this bound is optimal in the following sense: If one deformation retracts every amoeba $\cA(f_j)$ in $\cI(\cF)$ to its spine $\cS(f_j)$, then the number of connected components of the intersection of the deformed amoebas converges to $\MV(\New(f_1),\ldots,\New(f_n))$.
\label{Thm:AmoebaBernstein}
\end{Bernstein}

\begin{proof}
By Theorem \ref{Thm:SpineIntersectsComponents}  $ \cS(f_1)\cap_{\mathrm{st}}\cdots\cap_{\mathrm{st}}\cS(f_n)$ intersects each connected component of $\cI(\cF)$. Let $\Gamma$ be the subdivision of $\New(f_1) + \cdots + \New(f_n)$ which is induced by the subdivisions of $\New(f_1),\ldots,\New(f_n)$ which are pairwise dual to the spines $\cS(f_1),\ldots,\cS(f_n)$. By the Tropical Bernstein Theorem \ref{Thm:TropicalBernsteinTheorem}, we conclude that every element of the intersection $\cS(f_1)\cap_{\mathrm{st}}\cdots \cap_{\mathrm{st}} \cS(f_n)$ is dual to a mixed cell of $\Gamma$. By Lemma \ref{Lem:MixedCells} $\MV(\New(f_1),\ldots,\New(f_n))$ equals the sum of the volumes of the mixed cells of the induced subdivision in $\New(f_1) + \cdots + \New(f_n)$. Since our volume form is a lattice volume form induced by $\Z^n$, every mixed cell has at least volume one and the first part of the statement follows. The second part follows immediately from the Tropical Bernstein Theorem \ref{Thm:TropicalBernsteinTheorem}. 
The ``Moreover''-part is obvious.
\end{proof}

As a corollary of Theorem \ref{Thm:SpineIntersectsComponents} we also obtain the following B{\'e}zout type statement for the intersection of amoebas.

\begin{Bezout}
Let $\cF:=\{f_1,\ldots,f_n\}\subseteq \C[\mathbf{z}^{\pm 1}]$ be a generic collection of Laurent polynomials. For $1\leq j \leq n$, let $d_j := \deg(\trop_S(f_j))$. Let $c$ be the number of connected components of $\cI(\cF)$. Then,
$c\leq  \prod_{j = 1}^n d_j$. In particular, $c \leq \prod_{j = 1}^n \deg(f_j)$.
\label{Cor:AmoebaBezout}
\end{Bezout}

\begin{proof}
By Theorem \ref{Thm:SpineIntersectsComponents} we have the inequality $c \leq \# \bigcap_{j = 1}^n \cS(f_j)$. By the tropical B{\'e}zout theorem we conclude $c \leq \prod_{j = 1}^n d_j$; see e.g., \cite{Maclagan:Sturmfels}. By construction of the spine we know that $\cS(f_j)$ is dual to a regular subdivision of $\New(f_j)$ for every $1 \leq j \leq n$; see Section \ref{SubSec:Tropical}. 
Thus, $c \leq \prod_{j = 1}^n \deg(f_j)$.
\end{proof}

Though we have just seen that the number of connected components of $\cI(\cF)$ is bounded by the mixed volume and the degrees of the initial polynomials, it remains an open question, how many vertices a connected component of $\cI(\cF)$ can have. In particular, it is easy to see that the number of vertices $V(\cF)$ of $\cI(\cF)$ are not bounded by the dimension alone. Indeed, for any $m\in \N$ one can construct amoebas such that there exist components of the complement having $2m$ vertices. Figure \ref{Fig:BoundOfVerticesCounterexample} shows two such examples for $m=4$ and $m=8$. 

\begin{figure}
 \ifpictures
\includegraphics[width=0.4\linewidth]{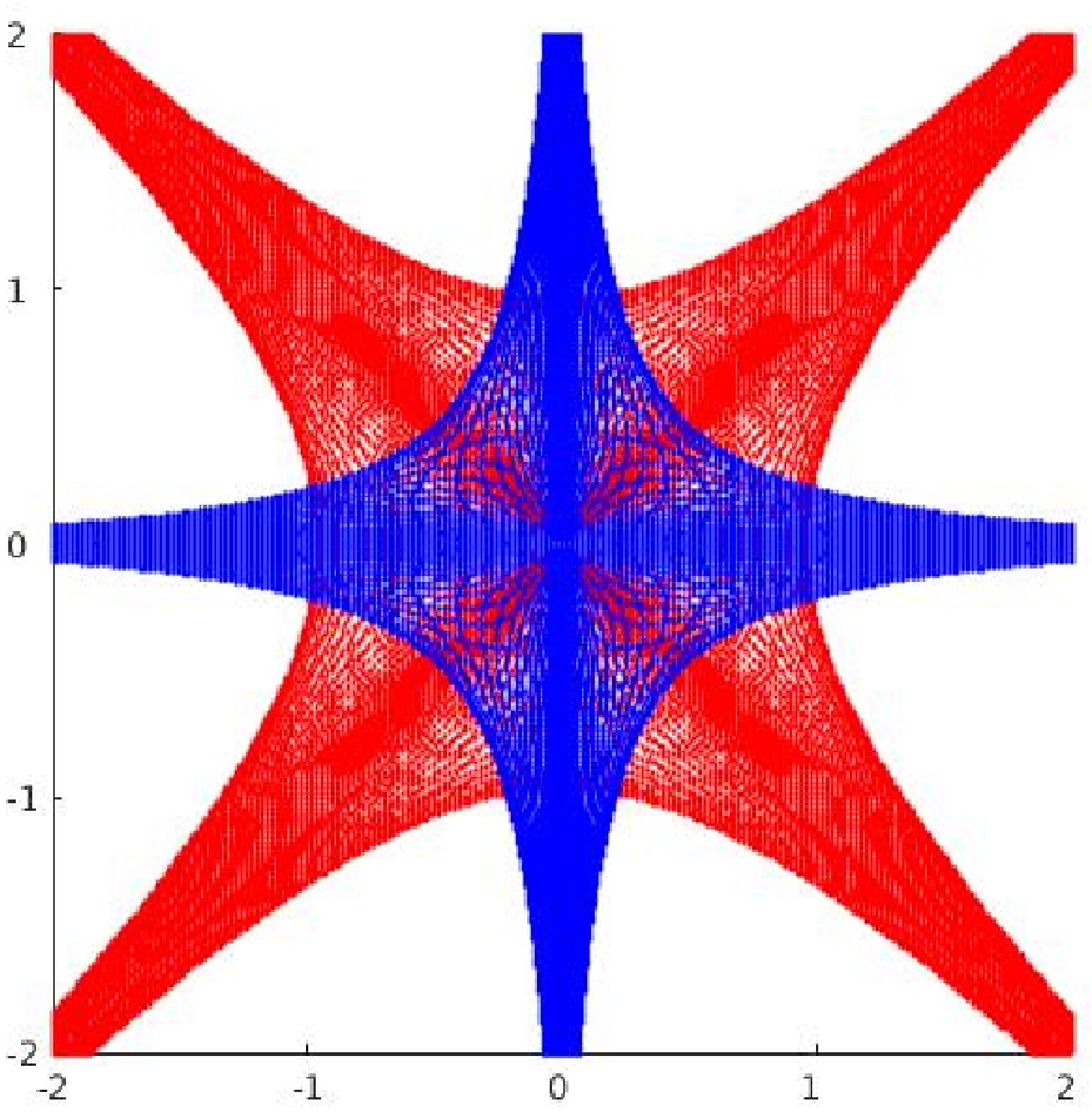} \quad
\includegraphics[width=0.4\linewidth]{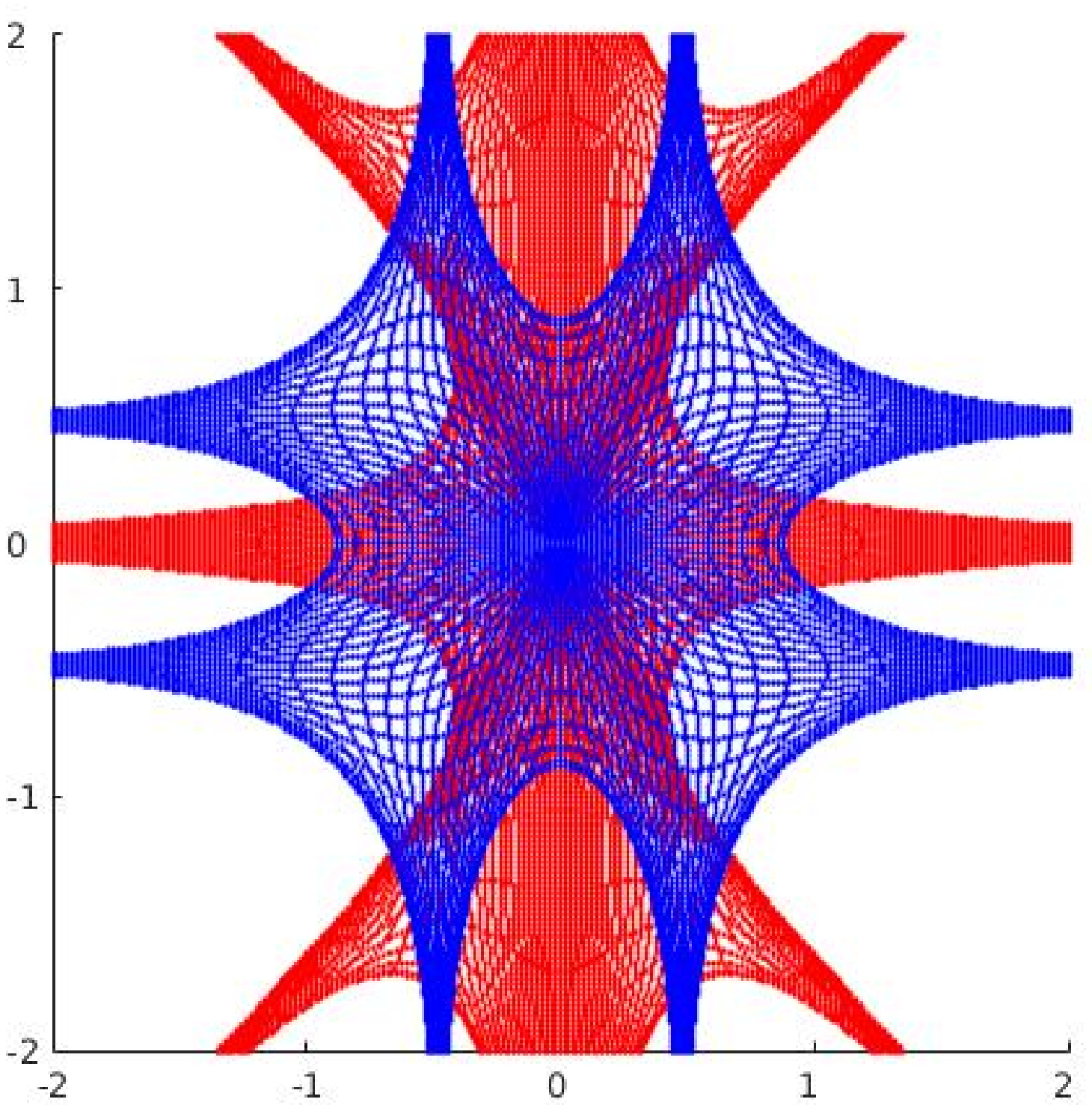} \quad
\fi
\caption{\emph{Left picture:} The intersection of $f_1(z_1,z_2) := z_1^2z_2^2 + z_1^2 + z_2^2 + 1$ and $f_2(z_1,z_2) := z_1^3 + z_2^3 + z_1 + z_2$. \emph{Right picture:} The intersection of $f_1(z_1,z_2) := z_1^3z_2^4 + z_1z_2^4 + 2z_1^4z_2^3 + 2z_2^3 + 2z_1^4z_2 + 2z_2 + z_1^3 + z_1$ and $f_2(z_1,z_2) :=  z_1^4z_2^4 + 3z_1^2z_2^4 + z_2^4 + 3z_1^4z_2^2 + 3z_2^2 + z_1^4 + 3z_1^2 + 1$. One can successively go on to construct intersections of amoebas with additional vertices in $V(\cF)$.}
\label{Fig:BoundOfVerticesCounterexample}
\end{figure}

\section{A Generalized Order Map}
\label{Sec:OrderMap}
In this section, we generalize the order map of an amoeba (see Section \ref{SubSec:OrderMap}) to intersections of generic collections of amoebas.
In order to do so, we first extend the usual order map of an amoeba $\cA(f)$ to all points on its boundary that lie in the closure of a \emph{unique} component of the complement of $\cA(f)$. For simplicity, we refer to those points as \struc{\textit{regular}} points in the following. If $p \in \ovl{E_{\alp}(f)} \cap \partial \cA(f)$ is regular, then we set $\struc{\ord(p)} := \alp$, i.e., the order of a point $p$ on the boundary of an amoeba $\cA(f)$ is the order of the unique component of the complement which contains $p$ in its closure.  We remark that amoebas with non-regular boundary points indeed do exist; see \cite[Figure 2, p. 58]{Rullgard:Diss} for an example.

From now on assume that $\cF:=\{f_1,\ldots,f_n\}\subseteq \C[\mathbf{z}^{\pm1 }]$ is a generic collection of Laurent polynomials. For $1\leq i\leq n$, we use \struc{$\ord_i$} to denote the just described extension of the order map of $f_i$. In order to simplify the notation further, we write $\New(\cF)$ for the Cartesian product of the Newton polytopes of $f_1,\ldots,f_n$, i.e.,
\begin{equation*}
\struc{\New(\cF)} \ := \ \New(f_1) \times \cdots \times \New(f_n).
\end{equation*}

Given these prerequisites, we generalize the order map of amoebas to intersections of amoebas in the following way: The \struc{\textit{generalized order map}} of $\cF$ is defined by

\begin{eqnarray*}
	\struc{\ord_{\cF}}: V(\cF) \to \Z^{n \times n} \qquad p \mapsto \left(\begin{array}{c} \ord_1(p) \\ \vdots \\ \ord_n(p) \\ \end{array}\right).
\end{eqnarray*}
  We refer to \struc{$\ord_{\cF}(p)$} as the \struc{\textit{order matrix}} of $p \in V(\cF)$. 
We want to remark that $\ord_{\cF}$ is well-defined, since, by condition (3) of genericity, every vertex of  $\cI(\cF)$ is regular for all amoebas in the intersection. 
Though the generalized order map $\ord_{\cF}$ is not injective in general, the next theorem shows that this is indeed the case if one restricts to a single intersection polytope.

\begin{thm}
Let $\cF:=\{f_1,\ldots,f_n\}\subseteq \C[\mathbf{z}^{\pm 1}]$ be a generic collection of Laurent polynomials. Let further $K$ be a connected component of $\cI(\cF)$ and let $P_K$ be the corresponding intersection polytope. Let $p$ and $q$ be vertices of $P_K$. Then 
\begin{equation*}
\ord_{\cF}(p)\neq \ord_{\cF}(q).
\end{equation*}
In particular, $\ord_{\cF}$ is injective on $V(P_K)$.
\label{Thm:OrderMapConnectivityComponents}
\end{thm}

\begin{proof}
In order to prove injectivity of the restriction of  $\ord_{\cF}$ to $V(P_K)$, we show that for vertices $p,q\in V(P_K)$ there exists an $\ell$ with $1\leq \ell \leq n$ and $\ord_\ell(p)\neq \ord_\ell(q)$. 

Assume  by contradiction that $\ord(p) = \ord(q)=:(\alpha(1),\ldots,\alpha(n))$, i.e., $\ord_i(p)=\ord_i(q)=\alpha(i)$ for $1\leq i\leq n$. 
Consider the line segment $\sig$ between $p$ and $q$ and let $v:=q-p$ be the corresponding vector pointing from $p$ to $q$. By assumption, Definition \ref{Def:VerticesOfComponents} and Theorem \ref{Thm:ExtremePoints}, we have $p,q\in \overline{E_{\alpha(i)}(f_i)}$ for $1\leq i\leq n$ and, by convexity of $E_{\alpha(i)}(f_i)$, we conclude that $\sig \subseteq \ovl{E_{\alpha(i)}(f_i)}$ for $1\leq i\leq n$. Hence,  $q - \eps \cdot v \in \overline{E_{\alpha(i)}(f_i)}$ for $1\leq i\leq n$ and $0\leq \eps \leq 1$. Moreover, since $q$ lies on the boundary of $\cA(f_i)$ for $1\leq i\leq n$ and $\overline{E_{\alpha(i)}(f_i)}$ is strictly convex, we infer that $q+\eps \cdot v\in \cA(f_i)$ for $\eps>0$ sufficiently small and $1\leq i\leq n$. In particular, $q+\eps \cdot v\in \cI(\cF)$ for $\eps>0$ sufficiently small, i.e., $q+\eps\cdot v\in P_K$ for $\eps >0$ sufficiently small. Hence, $q$ cannot be a vertex of $P_K$ and we obtain a contradiction. 
\end{proof}

Though, by Theorem \ref{Thm:OrderMapConnectivityComponents} the generalized order map is injective on each intersection polytope, this does not have to be true if one considers restrictions of the generalized order map to the vertex set of a connected component of $\cI(\cF)$. For instance, looking at Figure \ref{Fig:VerticesComponentsVersusConvexHulls} one sees that both connected components have  two vertices, that are not vertices of the corresponding intersection polytopes, whose orders are equal. Namely, they equal the order of the two bounded components of the complements of the amoebas. It also follows from the proof of Theorem \ref{Thm:OrderMapConnectivityComponents} that this example describes the only possible other case. More precisely, if $K$ is a connected component of $\cI(\cF)$ and if $\ord(p)= \ord(q)$ for two vertices $p,q\in V(K)$, then neither $p$ nor $q$ can be vertices of the corresponding intersection polytope $P_K$. 

The next proposition describes the normal cones of specific vertices of an intersection polytope. 

\begin{prop}
Let $\cF:=\{f_1,\ldots,f_n\}\subseteq \C[\mathbf{z}^{\pm 1}]$ be a generic collection of Laurent polynomials. Let $K$ be a connected component of $\cI(\cF)$ and let $P_K$ be the corresponding intersection polytope. Let $p\in V(P_K)$ be a vertex of $P_K$ such that $\ord_i(p)$ corresponds to a vertex $v_i$ of $\New(f_i)$ for $1\leq i\leq n$. Then the normal cone of $p$ in $P_K$ contains an affine translation of the intersection of all normal cones of the $v_i$ in $\New(f_i)$ for $1\leq i\leq n$, i.e., there exists a $v\in \R^n$ such that
\begin{equation*}
v+\bigcap_{i=1}^n\NF_{v_i}(f_i) \ \subseteq \ \NF_p(P_K).
\end{equation*}
\end{prop}

\begin{proof}
By assumption and Theorem \ref{Thm:OrderMapConnectivityComponents}, the vertex $p$ is the unique vertex of $P_K$, given by the intersection of $\ovl{E_{v_i}(f_i)}$ for $1\leq i\leq n$. As $v_i$ is a vertex of $\New(f_i)$, Theorem \ref{Thm:GKZAmoebaVertex} implies that the component $E_{v_i}(f_i)$ of the complement of $\cA(f_i)$ contains an affine translation of the normal cone $\NF_{v_i}(f_i)$. Hence, the intersection $\bigcap_{i = 1}^n E_{v_i}(f_i)$ contains an affine translation of $\bigcap_{i = 1}^n \NF_{v_i}(f_i)$. Moreover, the normal cone $\NF_p(P_K)$ has to contain all points given by $p + q$ that are not contained in any of the amoebas $\cA(f_i)$, i.e., $\bigcap_{i = 1}^n E_{v_i}(f_i) \subseteq \NF_p(P_K)$. Thus, we can conclude $v+\bigcap_{i=1}^n\NF_{v_i}(f_i)\subseteq \NF_p(P_K)$ for some $v\in \R^n$. 
\end{proof}

Since there is a relation between components of the complement of a single amoeba and the vertices of the corresponding Newton polytope (see Theorem \ref{Thm:GKZAmoebaVertex}) it is reasonable to ask if such a relation also exists if one considers intersections of amoebas and the Minkowski sum of their Newton polytopes. Before we can provide such a relation, we need to introduce some further notion.

For given Newton polytopes $\New(f_1),\ldots,\New(f_n)$ one defines the \struc{\textit{common refinement} $\NF(f_1,\ldots,f_n)$} of the normal fans $\NF(f_1),\ldots,\NF(f_n)$ as the fan given by all cones of the form $\bigcap_{j = 1}^n \NF_{G_j}(f_j)$ where $\NF_{G_j}(f_j)$ is an arbitrary cone of $\NF(f_j)$.

\begin{definition}
We call a non-empty, full dimensional cone $\bigcap_{j = 1}^n \NF_{G_j}(f_j)$ in the common refinement $\NF(f_1,\ldots,f_n)$ \struc{\emph{mixed}} if for all $1 \leq i \leq n$ it holds that
\begin{eqnarray*}
 \NF_{G_i}(f_i)\cap \bigcap_{j \in [n] \setminus \{i\}}\NF_{G_j}(f_j) & \neq & \NF_{G_i}(f_i)
\end{eqnarray*}  
and
\begin{eqnarray*}
 \NF_{G_i}(f_i)\cap \bigcap_{j \in [n] \setminus \{i\}}\NF_{G_j}(f_j) & \neq & \bigcap_{j \in [n] \setminus \{i\}}\NF_{G_j}(f_j).
\end{eqnarray*} 
\end{definition}

In the following we make use of the well-known fact \cite[Proposition 7.12.]{Ziegler:Book} that the normal fan of a Minkowski sum $\New(f_1) + \New(f_2)$ equals the common refinement of the individual normal fans $\NF(f_1)$ and $\NF(f_2)$.

Since we are aiming at a possible connection between $\cI(\cF)$ and the Minkowski sum of $\New(f_1),\ldots,\New(f_n)$, it is reasonable to consider not only a single intersection polytope but the convex hull of all intersection polytopes. In other words, we are interested in $\conv(V(\cF))$. It is an obvious question which points in $V(\cF)$ are actually vertices of $\conv(V(\cF))$. We are able to provide a characterization of the vertex set of $\conv(V(\cF))$ simply in terms of vertices of the Minkowski sum $\New(f_1)+\cdots +\New(f_n)$. To simplify the notation we write \struc{$\Mink(\cF)$} for this Minkowski sum in the following.

\begin{thm}
Let $\cF:=\{f_1,\ldots,f_n\}\subseteq \C[\mathbf{z}^{\pm 1}]$ be a generic collection of Laurent polynomials. Let $p\in V(\cF)$. 
 Then the following are equivalent:
 \begin{itemize}
  \item[(a)] $p$ is a vertex of $\conv(V(\cF))$.
  \item[(b)] For $1\leq j\leq n$ there exists a vertex $v:=v_1+\cdots +v_n\in \Mink(\cF)$, where $v_j\in \New(f_j)$ is a vertex, such that $p\in\bigcap_{j=1}^n \partial(\ovl{E_{v_j}(f_j)})$ and  $\NF_{v}(\Mink(\cF))$ is a mixed cone.
 \end{itemize}
\label{Thm:NormalFan}
\end{thm}
We remark that it follows from the proof of Theorem \ref{Thm:NormalFan} that in the situation of $(b)$, the intersection $\bigcap_{j=1}^n \partial(\ovl{E_{v_j}(f_j)})$ equals a single point, which is a vertex of $\conv(V(\cF))$.

\begin{proof}
Let $P$ denote the polytope that is given as the convex hull of $V(\cF)$, i.e., $P := \conv(V(\cF))$. Note that, by construction, the set of vertices of $P$ is a subset of $V(\cF)$. 

We first show ``$(b) \Rightarrow (a)$'':  
Let $v:=v_1+\cdots +v_n\in \Mink(\cF)$ be a vertex such that $v_j$ is a vertex of $\New(f_j)$ for $1\leq j\leq n$ and $C:=\NF_v(\Mink(\cF))$ is a full-dimensional mixed cone in the common refinement $\NF(\cF):=\NF(f_1,\ldots,f_n)$. Then it holds that $C=\bigcap_{j=1}^n \NF_{v_j}(f_j)$. 
Furthermore, by Theorem \ref{Thm:GKZAmoebaVertex} it follows that $\bigcap_{j = 1}^n \partial \ovl{E_{v_j}(f_j)} \neq \emptyset$. Since $C$ is mixed, we can conclude that $\bigcap_{j = 1}^n \partial \ovl{E_{v_j}(f_j)}$ has to be a single point. Indeed, for $n = 2$, this is obvious, since $\partial \ovl{E_{v_1}(f_1)} \cap \partial \ovl{E_{v_2}(f_2)} \neq \emptyset$ and $\NF_{v_1}(f_1) \cap \NF_{v_2}(f_2) \neq \NF_{v_1}(f_1),\NF_{v_2}(f_2)$. For arbitrary $n$, the claim follows from the same argument applied to every $\partial \ovl{E_{v_i}(f_i)}$, the one-dimensional subset given by $\bigcap_{j \in [n] \setminus \{i\}} \partial \ovl{E_{v_j}(f_j)}$, and the corresponding cones using the assumption that $C$ is mixed. 
Hence, $\bigcap_{j = 1}^n \partial \ovl{E_{v_j}(f_j)}$ equals a point $p$.  By Theorem \ref{Thm:GKZAmoebaVertex} an affine translation of $C$ is contained in $\bigcap_{j = 1}^n  \ovl{E_{v_j}(f_j)}$ and therefore $C$ cannot intersect $P$ in any point but $p$. Thus, we have $(p + C) \cap P = p$ and hence $p$ is a vertex of $P$. 

Next, we show that ``$(a) \Rightarrow (b)$''. 
Let $p$ be a vertex of $P$. It follows from Lemma \ref{Lem:FaceConstruction} that there exist components $E_{\alpha(1)}(f_1),\ldots,E_{\alpha(n)}(f_n)$ of the complements of the amoebas $\cA(f_1),\ldots,\cA(f_n)$ such that $p\in \bigcap_{j=1}^n\partial \ovl{E_{\alpha(j)}(f_j)}$. We prove the claim by contradiction. First assume that $\alpha(1),\ldots,\alpha(n)$ are vertices of $\New(f_1),\ldots,\New(f_n)$ but $\NF_{v}(\Mink(\cF))$ is not a mixed cone for $v:=\alpha(1)+\cdots +\alpha(n)$. Then there exists $1 \leq i \leq n$ such that $\NF_{\alpha(i)}(f_i) \subseteq \bigcap_{j \in [n] \setminus \{i\}}\NF_{\alpha(j)}(f_j)$. Thus, two possible cases concerning the behavior of the corresponding components of the amoebas' complements may occur:

{\sf Case 1:} $\ovl{E_{\alpha(i)}(f_i)} \subset \bigcap_{j \in [n] \setminus \{i\}} E_{\alpha(j)}(f_j)$. But then $p$ cannot be a vertex of $P$ since, by Lemma \ref{Lem:FaceConstruction}, all the $f_i$ need to be distinct, so as to obtain a vertex as the intersection of the boundary of $n$ non-redundant components $\ovl{E_{\alpha(j)}(f_j)}$.

{\sf Case 2:} $\ovl{E_{\alpha(i)}(f_i)} \setminus \bigcup_{j \in [n] \setminus \{i\}} \partial \ovl{E_{\alpha(j)}(f_j)}$ consists of two connected components $S_1, S_2$ satisfying that $S_1$ is contained in $\bigcap_{j \in [n] \setminus \{i\}} E_{\alpha(j)}(f_j)$ and $S_2$ is contained in the complement of $\bigcap_{j \in [n] \setminus \{i\}} \ovl{E_{\alpha(j)}(f_j)}$. 

Consider the set $T := \bigcap_{j \in [n] \setminus \{i\}} \partial \ovl{E_{\alpha(j)}(f_j)}$. Assume first that $S_2 = \emptyset$, i.e., $\ovl{E_{\alpha(i)}(f_i)} \subseteq \bigcap_{j \in [n] \setminus \{i\}}\ovl{E_{\alpha(j)}(f_j)}$ and $\partial\ovl{E_{\alpha(i)}(f_i)}$ intersects $T$ in a single point. 
Since this implies $E_{\alpha(i)}(f_i) \subseteq \bigcap_{j \in [n] \setminus \{i\}}\ovl{E_{\alpha(j)}(f_j)}$, we conclude that $T$ intersects the boundaries of two components of the complement of $\cA(f_i)$ which are distinct from $E_{\alpha(i)}(f_i)$. Since boundaries of such components are of codimension 1 and $T$ is of codimension $n-1$ this intersection yield points $p_1, p_2$, which are contained in $P$. Since we have $T \subseteq  \partial \ovl{E_{\alpha(j)}(f_j)}$ for every $j \neq i$ and $ \partial \ovl{E_{\alpha(j)}(f_j)}$ is strictly convex, the subset of $T$ connecting $p_1$ and $p_2$ is contained in the interior of $P$. And since  $E_{\alpha(i)}(f_i) \subseteq\bigcap_{j \in [n] \setminus \{i\}} \ovl{E_{\alpha(j)}(f_j)}$ it follows that $p$ is contained in the interior of $P$. Thus, $p$ is not a vertex.

Now, assume $S_2 \neq \emptyset$. By construction of Case 2 we have $S_2 \subseteq \Int(P)$, $S_1 \not \subseteq P$, and $\partial \ovl{E_{\alpha(i)}(f_i)}$ intersects $T$ in two points $u$ and $w$. (Here and in the sequel, we use $\Int(A)$ to denote the interior of a set $A$.) We can assume that $u$ and $w$ are vertices of $P$ or otherwise we can argue as in the $S_2 = \emptyset$ case. Since both $u$ and $w$ are given by the intersection of the same components of the complements of $\cA(f_i)$ with $1 \leq i \leq n$ we have $\ord(u) = \ord(w)$. But that is impossible with an analog argument as in the proof of Theorem \ref{Thm:OrderMapConnectivityComponents}.

Finally, assume that $p\in \bigcap_{j=1}^n\partial \ovl{E_{\alpha(j)}(f_j)}$ is a vertex of $P$ and at least one $\ovl{E_{\alpha(i)}(f_i)}$ does not correspond to a vertex of $\New(f_j)$. Namely, in this case we can consider the same set $T$ as above and by convexity of $\ovl{E_{\alpha(i)}(f_i)}$ the intersection $T \cap \partial \ovl{E_{\alpha(i)}(f_i)}$ is of cardinality two. Moreover, both points of intersection have the same order, which, as in the proof of Theorem \ref{Thm:OrderMapConnectivityComponents}, again yields a contradiction.
\end{proof}
We remark that if $p\in \conv(V(\cF))$ is a vertex, then by the proof of Theorem \ref{Thm:NormalFan} we do not have  $p\in \bigcap_{j=1}^n \partial \ovl{E_{v_j}(f_j)}$ with $v_j$ as in Theorem \ref{Thm:NormalFan} (b) alone, but it even holds that $p=\bigcap_{j=1}^n \partial \ovl{E_{v_j}(f_j)}$. Using this fact, the following is an immediate consequence of Theorem \ref{Thm:NormalFan}.

\begin{cor}
Let $\cF:=\{f_1,\ldots,f_n\}\subseteq\C[\mathbf{z}^{\pm 1}]$ be a generic collection of Laurent polynomials. Let $P$ be the convex hull of the vertices of $\cI(\cF)$, i.e., $P:=\conv(V(\cF))$ and let $V(P)$ be the vertex set of $P$.  
Then, the restriction of the order map $\ord_{\cF}$ to $V(P)$ is injective.
\label{Cor:OrderMapInjectivityTotalConvexHull}
\end{cor}

\begin{proof}
Let $p$ be a vertex of $P$. Then, as remarked above, Theorem \ref{Thm:NormalFan} implies that there exist vertices $v_j\in \New(f_j)$, $1\leq j\leq n$, such that $p=\bigcap_{j=1}^n \partial \ovl{E_{v_j}(f_j)}$. 
Since $\ord_{\cF}(p)=(v_1,\ldots,v_n)$ the claim follows.
\end{proof}

In the following, we associate another canonical polytope to a generic collection of Laurent polynomials $\cF:=\{f_1,\ldots,f_n\}$. Let $\mathcal{V}$ denote the image of the vertex set of $\cI(\cF)$ under the order map, i.e., $\struc{\mathcal{V}}:=\ord_{\cF}(V(\cF))$. We call the convex hull of $\mathcal{V}$ the \struc{\textit{order polytope}} associated to $\cF$. We denote this polytope by \struc{$\cO(\cF)$}.

We remark that if one considers generic collections $\cF:=\{f_1,\ldots,f_k\}\subseteq\C[\bf{z}^{\pm 1}]$ of $k < n$ Laurent polynomials, then one can generalize the order map in a similar way as above. Indeed, if a point $p$ belongs to the intersection of the boundaries of $\cA(f_1),\ldots,\cA(f_k)$, then, as before, one finds unique components $E_{\alpha(1)}(f_1),\ldots,E_{\alpha(k)}(f_k)$ of amoeba complements such that $p$ lies in their closure. One can now define the order matrix of $p$ to be the matrix with rows $\alpha(1),\ldots, \alpha(k)$. In this case the order map is a map from $\bigcap_{i=1}^k\partial\cA(f_i)$ to $\Z^{k\times n}$. Moreover, one can then define the order polytope of $\cF:=\{f_1,\ldots,f_k\}$ as the convex hull of the image of the generalized order map. Since this generalization, however, is irrelevant for our purposes, we do not pursue this direction further. The following is an easy example of an order polytope in the case that $k<n$. 

\begin{example}
Consider the trivial case $\cF := \{f\}$. Then $\cO(\cF)$ is the convex hull of the image of the usual order map $\ord$ applied to $\R^n \setminus \cA(f)$. By Theorem \ref{Thm:OrderMap} $\cO(\cF)$ is a lattice polytope contained in $\New(f)$. And by Theorem \ref{Thm:GKZAmoebaVertex} it follows that $\cO(\cF) = \New(f)$.
\end{example}

If $\cF$, however, is a collection of more than one Laurent polynomial, then $\cO(\cF)$ becomes less trivial. We provide some initial statements about the general case.

\begin{thm}
Let $\cF:=\{f_1,\ldots,f_n\}\subseteq\C[\mathbf{z}^{\pm 1}]$ be a generic collection of Laurent polynomials. Then:
\begin{itemize}
\item[(a)] $\cO(\cF)$ is a lattice polytope and contained in $\New(f_1) \times \cdots \times \New(f_n)$.
\item[(b)] Let $k$ denote the number of mixed cones in the common refinement $\NF(f_1,\ldots,f_n)$. Then $\cO(\cF)$ and $\New(\cF)$ have a least $k$ vertices in common. Let $\ord_\cF(p)$ be one of these joint vertices and let $\NF_{\ord_\cF(p)}(\New(\cF))$ be its corresponding normal cone in of $\New(\cF)$.  Then $\NF_{\ord_\cF(p)}(\New(\cF))$ contains a real $(n\times n)$-matrix in the interior with all rows equal.
\end{itemize}
\end{thm}

\begin{proof}
\begin{asparaenum}
\item[(a)] Since $\ord_{\cF}(p)\in \Z^{n\times n}$, Theorem \ref{Thm:OrderMap} implies that $\cO(\cF)$ is a lattice polytope. Consider a point $p\in V(\cF)$. Theorem \ref{Thm:OrderMap} implies that $\ord_i(p)\in \New(f_i)$ for $1\leq i\leq n$ and hence $\ord_{\cF}(p)\in \New(\cF)$. Since, as a polytope, $\New(\cF)$ is convex, $\cO(\cF)$ is contained in $\New(f_1)\times \cdots \times \New(f_n)$.
\item[(b)] By Theorem \ref{Thm:NormalFan} and the succeeding remark, there exists a bijection between the mixed cones in the common refinement $\NF(f_1,\ldots,f_n)$ and the vertices of $\conv(V(\cF))$. Let $p := \bigcap_{j = 1}^n \ovl{E_{\alp(i)}(f_i)}$ be a vertex of  $\conv(V(\cF))$. By Theorem \ref{Thm:NormalFan} (b) every $\alp(i)$ is a vertex in $\New(f_i)$. Thus, $\ord_\cF(p) = (\alp(1),\ldots,\alp(n))$ is a vertex of $\New(f_1) \times \cdots \times \New(f_n)$. Since $\cO(\cF) \subseteq \New(f_1) \times \cdots \times \New(f_n)$ and $\ord_{\cF}(p)\in \cO(\cF)$, it follows that $\ord_\cF(p)$ is also a vertex of $\cO(\cF)$. Therefore, $\cO(\cF)$ and $\New(f_1) \times \cdots \times \New(f_n)$ have at least $k$ vertices in common.

Furthermore, by Theorem \ref{Thm:NormalFan}, the cone $\NF_{\alp(1) + \cdots + \alp(n)}(\Mink(\cF))$ corresponding to $p$ is a mixed cone. In particular, the intersection of the normal cones $\NF_{\alp(1)}(f_1),\ldots,\NF_{\alp(n)}(f_n)$ is a full-dimensional cone. Thus, there exists a vector $(v_1,\ldots,v_n) \in \Int(\bigcap_{j = 1}^n \NF_{\alp(j)}(f_j))$.
By construction the normal cone $\NF_{\ord(p)}(\New(\cF))$ is the Cartesian product of the normal cones $\NF_{\alp(1)}(f_1),\ldots,\NF_{\alp(n)}(f_n)$. Hence, $(v_1,\ldots,v_n) \in \Int(\bigcap_{j = 1}^n \NF_{\alp(j)}(f_j))$ if and only if
\begin{eqnarray*}
	\left(\begin{array}{ccc}
		v_1 & \cdots & v_n \\ \vdots & \ddots & \vdots \\ v_1 & \cdots & v_n \\
	\end{array}\right) & \in & \Int (\NF_{\ord(p)}(\New(\cF))).
\end{eqnarray*}
\end{asparaenum}
\end{proof}

\begin{example}
We consider once more the intersection $\cI(\cF)$ with $\cF$ given by $f_1(z_1,z_2) := z_1^2z_2^2 + z_1^2 + z_2^2 + 1$ and $f_2(z_1,z_2) := z_1^3 + z_2^3 + z_1 + z_2$. $\cI(\cF)$ is shown in the left picture of Figure \ref{Fig:BoundOfVerticesCounterexample}; $V(\cF)$ has 8 elements, which are all vertices in $\conv(V(\cF))$. Hence, all elements correspond to a unique vertex in the Minkowski sum; see Figure \ref{Fig:MinkowskiSum} . $\New(f_1) \times \New(f_2)$ is a polytope with 16 vertices given by the product of two rectangles in $\R^4$. It follows that $\cO(\cF)$ is the convex hull of 8 of these 16 vertices as shown in the Schlegel diagram in Figure \ref{Fig:SchlegelDiagram}.
\end{example}

\begin{figure}
 \ifpictures
  \includegraphics[width=0.32\linewidth]{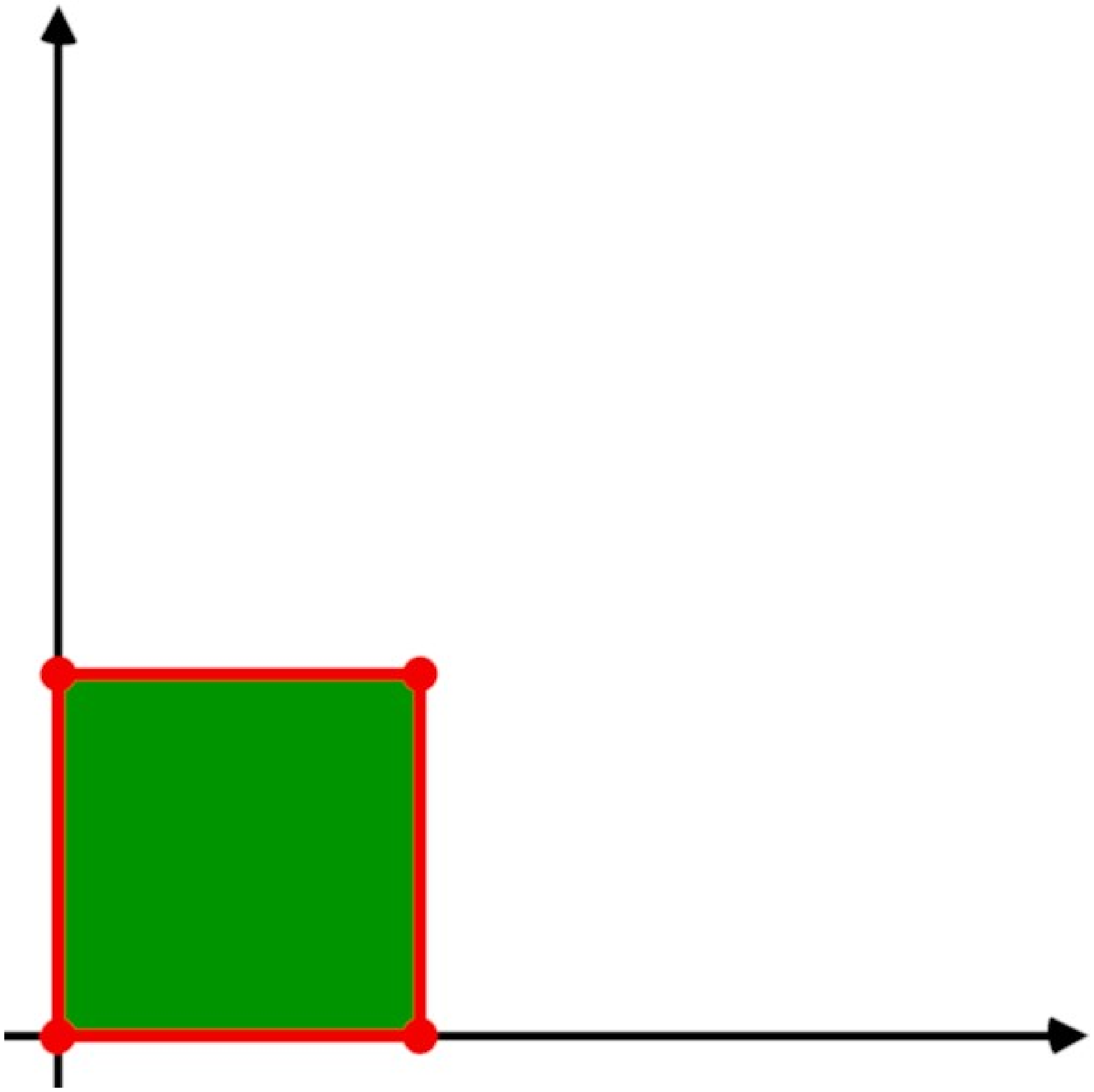}
  \includegraphics[width=0.32\linewidth]{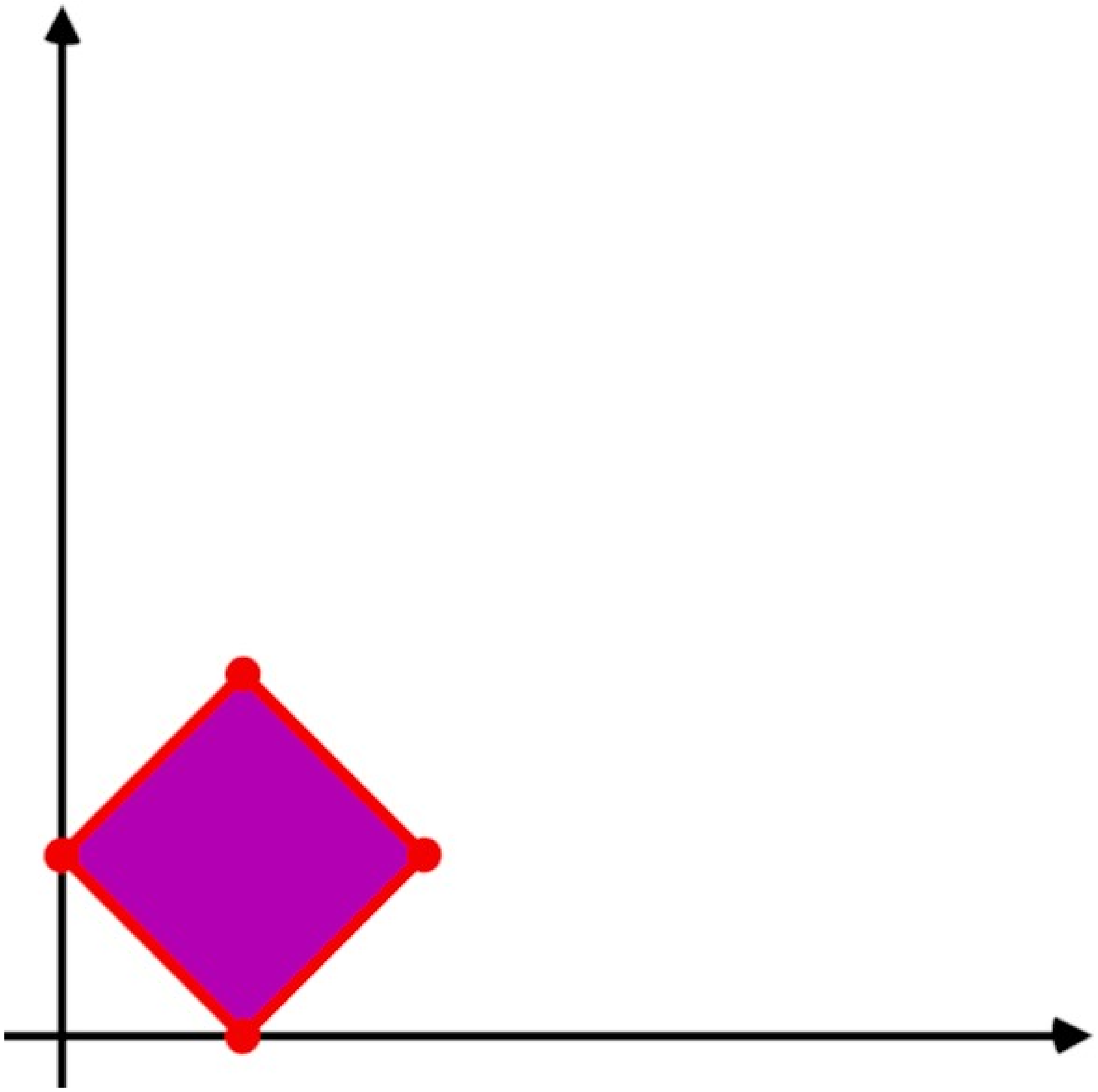}
  \includegraphics[width=0.32\linewidth]{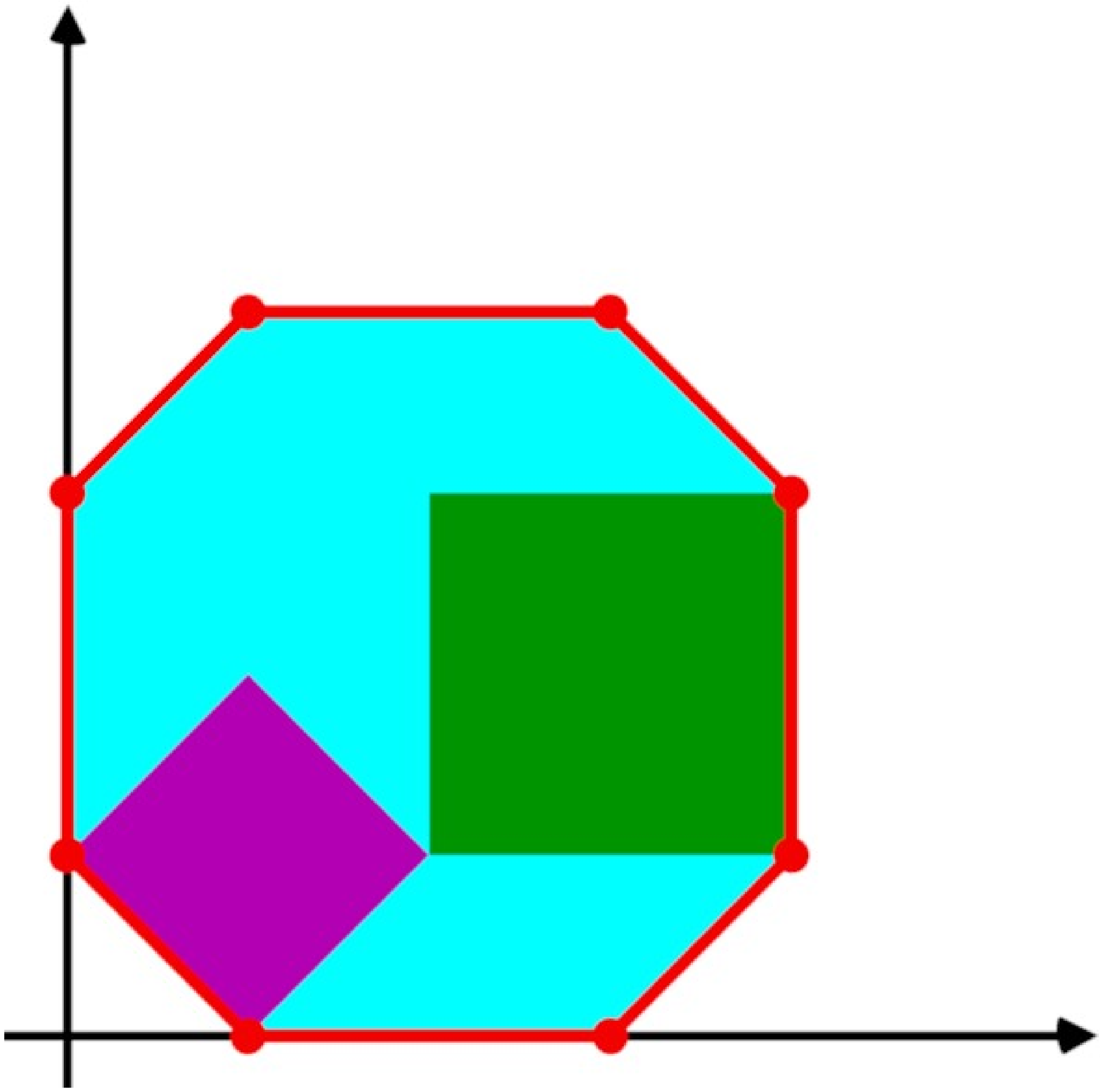}
 \fi
\caption{The Newton polytopes of $f_1(z_1,z_2) := z_1^2z_2^2 + z_1^2 + z_2^2 + 1$ and $f_2(z_1,z_2) := z_1^3 + z_2^3 + z_1 + z_2$ and their Minkowski sum.}
\label{Fig:MinkowskiSum}
\end{figure}

\begin{figure}
 \ifpictures
  \includegraphics[width=0.32\linewidth]{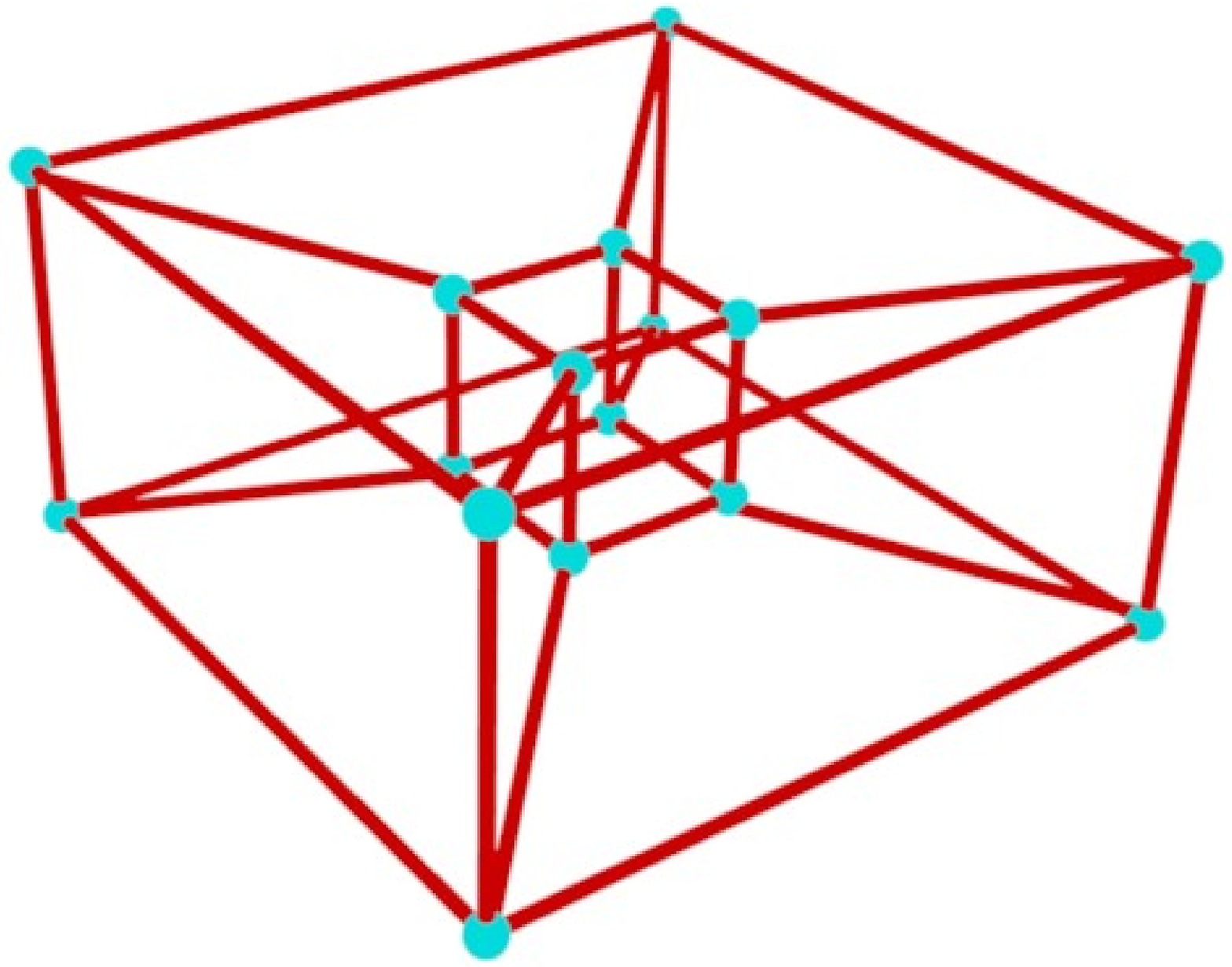}
  \includegraphics[width=0.32\linewidth]{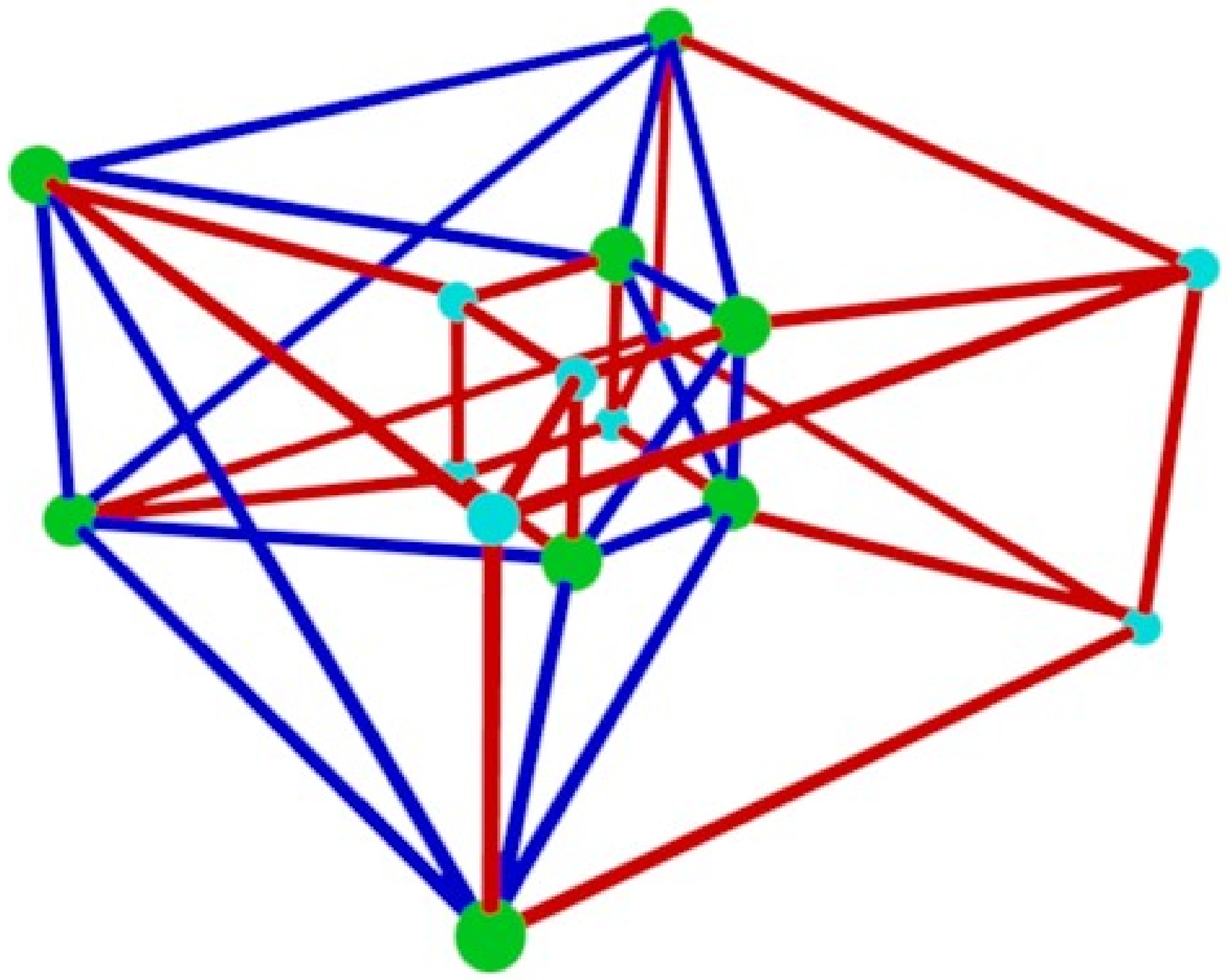}
  \includegraphics[width=0.32\linewidth]{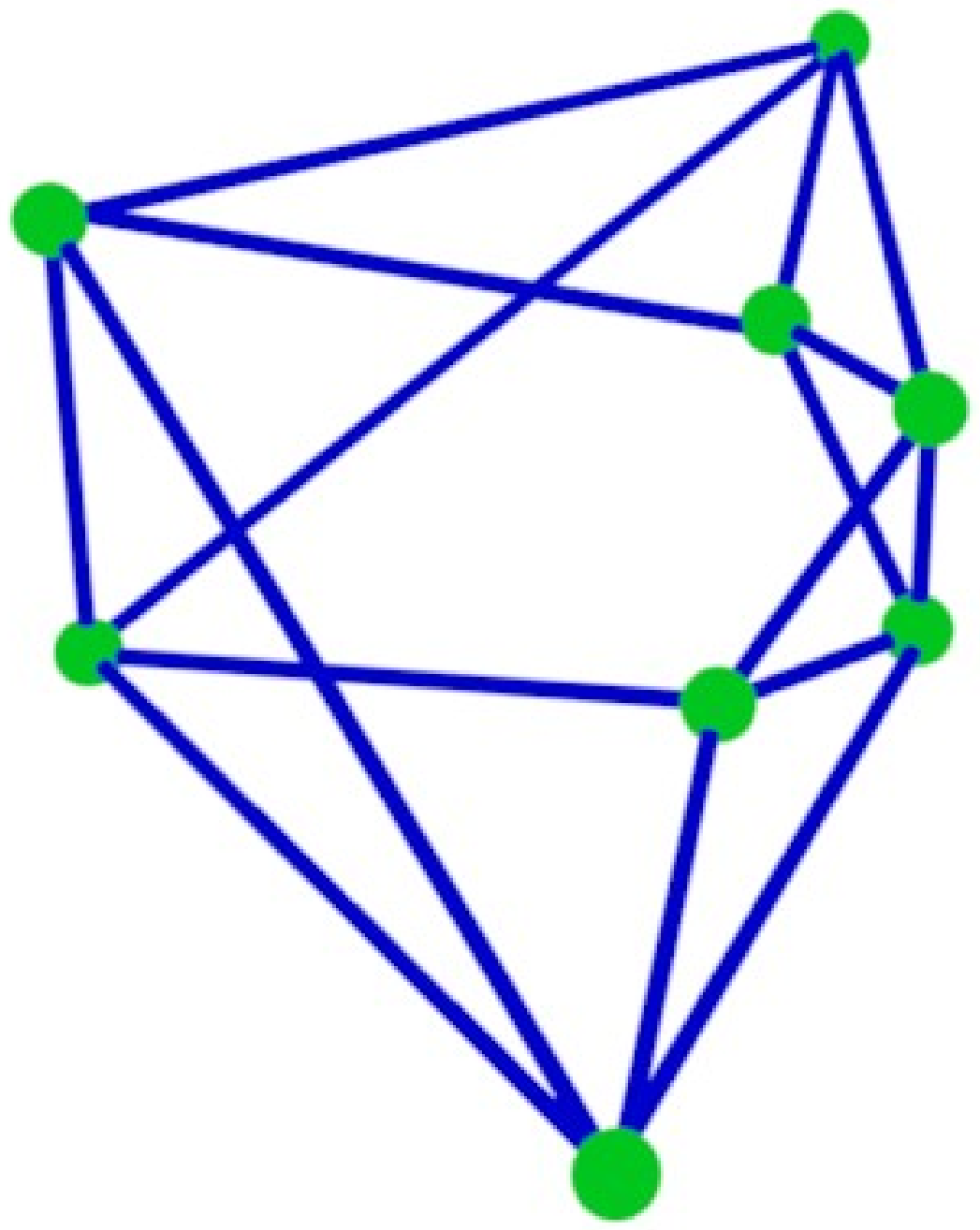}
 \fi
\caption{The Schlegel diagram of $\New(f_1) \times \New(f_2)$ with $f_1,f_2$ as in Figure \ref{Fig:MinkowskiSum} together with the order polytope $\cO(\cF)$ for $\cF := \{f_1,f_2\}$ both inside $\New(f_1) \times \New(f_2)$ and without $\New(f_1) \times \New(f_2)$.}
\label{Fig:SchlegelDiagram}
\end{figure}

\appendix

\section{Genericity}
\label{Sec:AppendixGenericity}

For a given $f \in \C[\mathbf{z}^{\pm 1}]$ with $\cV(f) \subseteq (\C^*)^n$ we define the \struc{\textit{unlog amoeba $\cU(f)$}} as the image of $\cV(f)$ under the map

\begin{eqnarray}
	\struc{|\cdot|}: \ \lf(\C^*\ri)^n \to \R_{> 0}^n, \quad (z_1,\ldots,z_n) \mapsto (|z_1|,
\ldots, |z_n|) \, .
\end{eqnarray}

\begin{lemma}
Let $f \in \C[\mathbf{z}^{\pm 1}]$. The unlog amoeba $\cU(f)$ is semi-algebraic and hence its boundary is (real) algebraic.
\label{Lem:UnlogIntersectionAlgebraic}
\end{lemma}

\begin{proof}
First, $\cU(f)$ is closed since $\cU(f)$ is isomorphic to $\cA(f)$ via a componentwise logarithm. Second, we have
$$\cU(f) \ = \ \{|\mathbf{z}| \in \R_{>0}^{n} \ : \ \exists \theta \in [0,2\pi): f(|\mathbf{z}| \cdot e^{i \cdot \theta}) = 0\}.$$
We consider the \struc{\textit{realification}} of $f(\mathbf{z})$, i.e., we consider the real and imaginary parts $\struc{f^{\re}},\struc{f^{\im}} \in \R[\mathbf{x},\mathbf{y}]$ such that 
\[ 
  f(\mathbf{z}) \ = \ 
  f(\mathbf{x}+i\mathbf{y}) \ = \ f^{\re}(\mathbf{x},\mathbf{y}) + i \cdot
f^{\im}(\mathbf{x},\mathbf{y})\,,
\]
see also \cite{Schroeter:deWolff,Theobald:deWolff:SOS}. 
Consider the real semi-algebraic set $S$ given by
\begin{eqnarray*}
f^{\re}(|z_1| \cdot u_1,\ldots,|z_n|\cdot u_n,|z_1| \cdot v_1,\ldots,|z_n|\cdot v_n) & = & 0 \\
f^{\im}(|z_1| \cdot u_1,\ldots,|z_n|\cdot u_n,|z_1| \cdot v_1,\ldots,|z_n|\cdot v_n) & = & 0 \\
u_j^2 + v_j^2 & = & 1 \text{ for all } j = 1,\ldots,n \\
|z_j| & > & 0 \text{ for all } j = 1,\ldots,n. 
\end{eqnarray*}
Then $\cU(f)$ is the projection of $S$ on the variables $|z_1|,\ldots,|z_n|$. Hence, $\cU(f)$ is semi-algebraic by the Tarski-Seidenberg principle; see \cite[Proposition 2.83]{Bochnak:Coste:Roy:RealAlgebraic}. Thus, its boundary is algebraic.
\end{proof}

With this lemma we can justify the term ``generic'' for the genericity condition (2) in Section \ref{Sec:Combinatorics}.

\begin{cor}
Let $A_1,\ldots,A_n\subseteq \Z^n$ be fixed support sets. There exists a Zariski open set $S\subseteq (\C^*)^{A_{1}} \times \cdots \times (\C^*)^{A_{n} }$ such that for $(f_{1},\ldots,f_n)\in S$ and all $1\leq i_1<\cdots <i_k\leq n$ the intersection $\bigcap_{\ell=1}^k \partial(\cA(f_{i_\ell}))$ has codimension $k$.
\end{cor}

\begin{proof}
As the intersection of a finite number of Zariski open sets, is Zariski open, it suffices to show that, for all $1\leq i_1<\cdots<i_k\leq n$, there exists a Zariski open set $S$ of $(\C^*)^{A_{i_1}} \times \cdots \times (\C^*)^{A_{i_k}}$ such that $\bigcap_{\ell=1}^k \partial(\cA(f_{i_\ell}))$ has codimension $k$ for all $(f_{i_1},\ldots,f_{i_k})\in S$. 
 Since the intersection of the boundaries of the unlog amoebas $\cU(f_{i_1}),\ldots,\cU(f_{i_k})$ is algebraic by Lemma \ref{Lem:UnlogIntersectionAlgebraic}, there exists a Zariski open set $S$ of $(\C^*)^{A_{i_1}} \times \cdots \times (\C^*)^{A_{i_k}}$ such that $\bigcap_{\ell = 1}^k \partial(\cU(f_{i_\ell})))$ has codimension $k$ for all $(f_{i_1},\ldots,f_{i_k})\in S$. Since $\Log : (\R^*)^n \to \R^n$ is a  diffeomorphism, $\Log(\bigcap_{\ell = 1}^k \partial(\cU(f_{i_\ell})))$ has codimension $k$ on $S$. Moreover, as  diffeomorphisms  are preserved under intersections and taking boundaries, we conclude that
  $$ \bigcap_{\ell=1}^k \partial(\cA(f_{i_\ell}))= \bigcap_{\ell = 1}^k \Log\left(\partial(\cU(f_{i_\ell}))\right)= \Log\left(\bigcap_{\ell = 1}^k \partial(\cU(f_{i_\ell}))\right)$$
   has codimension $k$  for all $(f_{i_1},\ldots,f_{i_k})\in S$.
\end{proof}

Now, we show that genericity condition (3) in Section \ref{Sec:Combinatorics} is also satisfied on an open set.

\begin{lemma}
Let $A_1,\ldots,A_n\subseteq \Z^n$ be fixed support sets. Let $W$ be the set of all Laurent polynomials $(f_1,\ldots,f_n) \in (\C^*)^{A} := (\C^*)^{A_1} \times \cdots \times (\C^*)^{A_n}$ such that each $p \in \bigcap_{j = 1}^n \cA(f_j)$  is contained in the closure of at most one component of the complement of $\cA(f_k)$ for $1\leq k\leq n$. Then, the  set $W$ is a non-empty Zariski open subset of $(\C^*)^A$.
\end{lemma}

\begin{proof}
By a theorem of Rullg{\aa}rd \cite[Thm. 13, p. 36]{Rullgard:Diss} there exists an $f_i \in (\C^*)^{A_i}$ for every $i = 1,\ldots,n$ such that the corresponding amoeba $\cA(f_i)$ is \textit{\struc{maximally sparse}}, i.e., the number of connected components of the complement of $\cA(f_i)$ equals the number of vertices of $\New(f_i)$. Recall that this is the minimal number of possible connected components; see Theorem \ref{Thm:GKZAmoebaVertex}. By Theorem \ref{Thm:GKZAmoebaVertex} and the convexity of every connected component of the complement of $\cA(f_i)$, it also follows that a maximally sparse amoeba does not contain points in its boundary, which are contained in the closure of more than one component of the complement. Since, by a theorem of Forsberg, Passare, and Tsikh \cite[Prop. 1.2]{Forsberg:Passare:Tsikh}, components of the complement of an amoeba are lower semi-continuous under a change of coefficients, there exists an open (with respect to the standard topology) $\eps$-ball $B_\eps(f_i)$ around $f_i$ in $(\C^*)^{A_i}$ such that for any $g\in B_\eps(f_i)$ the amoeba $\cA(g)$ does not have points on its boundary lying in more than one complement component. Since we can apply the same argument to any $f_i$, we find an open (with respect to standard topology) $\eps$-ball $B_{\eps}$ around $(f_1,\ldots,f_n) \in (\C^*)^{A}$ such that each $p \in \bigcap_{j = 1}^n \cA(f_j)$  is contained in the closure of at most one component of the complement of $\cA(f_k)$ for $1\leq k\leq n$. Now, we consider  the Zariski closure $\overline{P}$ of the set $P$ of all families of polynomials $(g_1,\ldots,g_n) \in (\C^*)^{A}$ such that there exists a $p \in \bigcap_{j = 1}^n \cA(g_j)$ which is contained in the closure of to components of the complement of some $\cA(g_k)$. Then $\overline{P}$ is not equal to $(\C^*)^A$, since it does not contain $B_{\eps}$. Thus, the set $W$ is a Zariski open subset of $(\C^*)^A$.
\end{proof}

\section{Dimension}
\label{Sec:AppendixDimension}

It is well-known that the boundary of a single amoeba $\cA(f)$ is contained in the $\Log|\cdot|$-image of all points in $\cV(f)$, which are critical under the $\Log|\cdot|$-map. This image is called the \textit{\struc{contour}} of $\cA(f)$ and denoted by  $\struc{\cC(f)}$.  
Since the boundary of a single amoeba $\cA(f)$ is contained in the contour $\struc{\cC(f)}$ in $\R^n$, which is closed real-analytic hypersurface in $\cA(f) \subset \R^n$ (see \cite{Passare:Tsikh:Survey}), we have a canonical notion of \textit{\struc{dimension}} (using the standard topology in $\R^n$). This notion of dimension extends to intersection of amoebas in $\R^n$ in the usual way for intersections of analytical sets.

\bibliographystyle{plain}
\bibliography{AmoebaIntersection}

\end{document}